%% file: Adaptive_AC_2D_with_Simplified_Estimator.tex
\documentclass[review]{elsarticle}

\usepackage{lineno,hyperref}
\modulolinenumbers[5]

\journal{}

\usepackage{graphicx, epsfig, psfrag}
\usepackage{amsfonts,amsmath,amssymb,amsbsy,amsthm}
\usepackage{bm}
\usepackage{color}
\usepackage{xcolor}
\usepackage{float}
\usepackage{hyperref}
\usepackage{mathrsfs}
\usepackage{longtable}
\usepackage[top=3cm,bottom=3cm,left=3cm,right=3cm]{geometry}
\usepackage{epstopdf}
\usepackage{stmaryrd}
\usepackage{ulem}
\usepackage{algorithm}
\usepackage{subfig}
\usepackage{enumitem}

\newtheorem{theorem}{Theorem}
\newtheorem{corollary}[theorem]{Corollary}
\newtheorem{assumption}{Assumption}

\newtheorem{remark}{Remark}

\definecolor{yscol}{HTML}{6622AA}

\definecolor{hwcol}{rgb}{0, 0, 0.9}
\definecolor{mlcol}{rgb}{0, 0.7, 0}
\definecolor{todocol}{rgb}{0.0, 0.4, 0.0}

\input{notation}

\begin{document}

\begin{frontmatter}

\title{Efficient a Posteriori Error Control of a Consistent Atomistic/Continuum Coupling Method for Two Dimensional Crystalline Defects}


\author[mymainaddress]{Yangshuai Wang}

\author[mysecondaryaddress]{Hao Wang\corref{mycorrespondingauthor}}
\cortext[mycorrespondingauthor]{Corresponding author}
\ead{wangh@scu.edu.cn}

\address[mymainaddress]{Department of Mathematics, University of British Columbia, 1984 Mathematics Road, Vancouver, Canada}
\address[mysecondaryaddress]{School of Mathematics, Sichuan University, No. 24 Yihuan Road, Chengdu, China}

\begin{abstract}
Adaptive atomistic/continuum (a/c) coupling method is an important method for the simulation of material and atomistic systems with defects to achieve the balance of accuracy and efficiency. Residual based {\it a posteriori} error estimator is often employed in the adaptive algorithm to provide an estimate of the error of the strain committed by applying the continuum approximation for the atomistic system and the finite element discretization in the continuum region. In this work, we propose a theory based approximation for the residual based {\it a posteriori} error estimator which greatly improves the efficiency of the adaptivity. In particular, the numerically expensive modeling residual is only computed exactly in a small region around the coupling interface but replaced by a theoretically justified approximation by the coarsening residual outside that region. We present a range of adaptive computations based on our {\it modified} {\it a posteriori} error estimator and its variants for different types of crystalline defects some of which are not considered in previous related literature of the adaptive a/c methods. The numerical results show that, compared with the {\it original} residual based error estimator, the adaptive algorithm using the {\it modified} error estimator with properly chosen parameters leads to the same optimal convergence rate of the error but reduces the computational cost by one order with respect to the number of degrees of freedom.

%
%
%
%
%
\end{abstract}

\begin{keyword}
Atomistic-to-continuum coupling, crystalline defects, a posteriori error estimate, 
adaptivity.
\MSC[2020] 65N12, 65N15, 70C20, 82D25
\end{keyword}

\end{frontmatter}

\input{introduction}

\input{formulation}
\input{errorAnal}

\input{numerics}

\input{conclusion}

\appendix
\renewcommand\thesection{\appendixname~\Alph{section}}
\input{appendix}


\bibliography{qc}

\end{document}

%% file: notation.tex
\def\XXint#1#2#3{{\setbox0=\hbox{$#1{#2#3}{\int}$ }
\vcenter{\hbox{$#2#3$ }}\kern-.6\wd0}}

\def\b{\big}
\def\B{\Big}
\def\bg{\bigg}

\def\setsep{\,|\,}
\def\bsetsep{\,\b|\,}

\def\R{\mathbb{R}}

\def\Z{\mathbb{Z}}
\def\C{\mathbb{C}}

\def\dx{\,{\rm d}x}

\def\dr{\,{\rm d}r}

\def\pp{\partial}

\def\<{\langle}
\def\>{\rangle}

\def\mA{\mathsf{A}}

\def\mF{{\sf F}}

\def\mJ{{\sf J}}

\def\mR{{\sf R}}

\def\tr{{\sf tr}}

\newcommand{\Dc}[1]{\D_{#1}}
\def\D{\nabla}
\def\del{\delta}
\def\ddel{\delta^2}

\def\a{{\rm a}}

\def\c{{\rm c}}

\def\a{{\rm a}}
\def\c{{\rm c}}
\def\ac{{\rm ac}}
\def\i{{\rm i}}

\def\L{\Lambda}

\def\Cs{\mathcal{C}}

\def\Us{\mathscr{U}}

\def\Ush{\dot{\Us}^{1,2}}

\def\Usc{\Us^c}

\def\E{\mathscr{E}}
\def\Ea{\E^\a}
\def\Eb{\E^{\rm b}}

\def\Ec{\E^\c}
\def\Eh{\E^\ac}
\def\F{\mathscr{F}}

\def\L{\Lambda}

\def\Li{\L^{\i}}

\def\Nhd{\mathcal{N}}

\def\Rg{\mathscr{R}}

\def\vsig{\varsigma}
\def\T{\mathcal{T}}
\def\Tp{T^+}
\def\Tm{T^-}
\def\np{\nu^+}
\def\nm{\nu^-}
\def\Ta{\T_\a}
\def\Th{\mathcal{T}_h}

\def\Fh{\mathscr{F}_h}
\def\UsT{\Us_h}

\def\Ia{I_\a}

\def\vor{\rm vor}
\def\s{\sigma}
\def\sh{\sigma^{\rm ac}}
\def\sa{\sigma^\a}
\def\sigc{\sigma^\c}

\def\PO{{\rm P}_0}

\def\Rg{\mathcal{R}}

\def\Lhom{\L^{\rm hom}}

\def\trc{{\rm tr}}
\def\mo{{\rm mo}}
\def\cg{{\rm cg}}

%% file: introduction.tex
\section{Introduction}
\label{sec:introduction}

Atomistic/continuum (a/c) coupling methods are a class of concurrent multi-scale
schemes coupling molecular mechanics models of atomistic processes with continuum models of long-ranged elastic fields \cite{van2020roadmap, TadmorMiller:2012, LuOr:acta, Ortiz:1995a, Miller:2008, olson2014optimization, olson2016analysis}. The goal of the a/c methods is to combine the accuracy of the atomistic models and the efficiency of the continuum models for the simulation of atomistic systems, especially those with defects. In particular, the atomistic model is often applied in a small neighborhood of the localized defects such as vacancies and dislocations, while the continuum model is employed away from the defect cores. We refer to \cite{LuOr:acta} for an extensive overview and benchmarking of different a/c methods.

In the past two decades, a/c coupling methods have attracted considerable attention from both the
engineering community and the mathematical community \cite{LuOr:acta, 2013-defects, LiOrShVK:2014, OrZh:2016, 2011_BD_LC_TO_SP_CMAME, davis2022moving, chakraborty2021crystal}.
The quantitative estimates of the approximation error for the a/c coupling methods motivate and establish analytical frameworks for similar multiscale computational
methods \cite{2006_TO_SP_AR_PM_SISC, oden_2018, chakraborty2021concurrent, gupta2021nonequilibrium}.
The rigorous mathematical and numerical analysis for the a/c coupling methods provides the theoretical foundation for achieving (quasi-)optimal balance between accuracy and computational cost for the simulation of various crystalline defects \cite{olson2014optimization, olson2016analysis, LiOrShVK:2014, OrZh:2016, emingyang, OrtnerWang:2011, MiLu:2011, PRE-ac.2dcorners}.

While the modeling and the {\it a priori} analysis of a/c coupling methods have received extensive and comprehensive discussion by numerous previous works \cite{LuOr:acta, olson2014optimization, LiOrShVK:2014, OrtnerWang:2011, LinP:2006a}, only a few researches concern the {\it a posteriori} analysis and adaptivity which is the cornerstone for the applications of the methods in real world material system. 

The goal-oriented {\it a posteriori} error estimate is the first approach applied to the adaptive a/c problems. Oden and Prudhomme propose a general framework for the estimation of modeling error in computational mechanics in \cite{2003_TO_SP_JCP} which is later applied to the adaptive computation of a practical three dimensional nanoindentation problem in \cite{2006_SP_PB_TO_IJMCE} together with a fine discussion and an extension to the problem of molecular dynamics in \cite{2006_TO_SP_AR_PM_SISC}. Arndt and Luskin derive the goal-oriented error estimator for the one-dimensional Frenkel-Kontorova model in a series of works \cite{2007_MA_ML_IJMCEpdf, 2008_MA_ML_MMS, 2008_MA_ML_CMAME} where a detailed study of the model adaptivity and the mesh adaptivity using the goal-oriented approach is given. However, the a/c methods employed in the works just mentioned are the original energy-based quasicontinuum method (QCE) which, despite its comparatively easy implementation, is inconsistent and suffers from the ghost-force at the atomistic to continuum coupling interface \cite{LuOr:acta, OrtnerWang:2011, OrtnerShapeev:2011}. More recent advances in this direction are given in \cite{2011_BD_LC_TO_SP_CMAME} and \cite{2009_SP_LC_BD_PB_CMAME} where an adaptive modeling error control for an a/c method based on the Arlequin framework \cite{2004_BD_GR_IJNME} are analyzed. The advantage of the goal-oriented {\it a posteriori} error estimate is that it concentrates on the minimization of the error of some prescribed quantities of interests which often involve certain geometric parameters of the defect core, though the exact value of such quantities may not be easy to estimate so that a proper stopping criterion for the adaptive algorithm is yet to be proposed \cite{2011_BD_LC_TO_SP_CMAME}. In contrast, the residual based {\it a posteriori} error estimate, which essentially follows the classic adaptive finite element methods, provides a quantitative estimate of the error in certain global norms that are related more directly to the strain and the energy of the system \cite{Ortner:qnl.1d, OrtnerWang:2014, wang2018posteriori}. The residual based {\it a posteriori} error control is analyzed in \cite{Ortner:qnl.1d, OrtnerWang:2014} for consistent a/c coupling methods in one-dimension. The extension to two dimensional system with nearest neighbor interaction is first carried out in \cite{wang2018posteriori} where {\it rigorous} {\it a posteriori} estimates of the residual and the stability constant are derived. The analysis in \cite{wang2018posteriori} also clarifies the so-called ``stress tensor correction" which essentially distinguishes the high dimensional results with previous one dimensional ones. With the {\it a posteriori} error estimates and the corresponding adaptive algorithm, the numerical experiments in \cite{wang2018posteriori} achieve the adaptive movement of the a/c interface and the
mesh refinement in the continuum region. The two dimensional analysis is then extended to the case of finite range interactions in \cite{liao2018posteriori} by introducing a computable approximated stress tensor formulation. 
However, both of the two works consider only simple point defects (di-vacancy in \cite{wang2018posteriori} while di-vacancy and micro-crack in \cite{liao2018posteriori}). More importantly, the main drawback of the {\it original} residual based {\it a posteriori} error estimates proposed in these two works is the inefficiency of the resulting adaptive algorithms which comes from the high computational cost of the evaluation of the modeling residual. The numerical experiments in \cite{liao2018posteriori} find that the CPU time of computing the modeling residual dramatically exceeds that of solving the a/c coupling problem.

The purpose of this work is to develop an efficient adaptive algorithm based on a theory based approximation for the residual based {\it a posteriori} error estimator for a consistent a/c coupling scheme and apply the algorithm to the adaptive simulations of different defects in two dimensions. In particular, we prove that the modeling residual, which is expensive to compute, and the coarsening residual, whose computational cost is marginal compared with that of solving the a/c problem on a given mesh, essentially follow a ratio relation under certain assumptions. Based on this theoretical result, we then propose various {\it modified} residual based {\it a posteriori} error estimator and mesh structures so that the modeling residual is only exactly computed in a relatively small neighborhood of the coupling interface but properly approximated by the coarsening error away from the defect core. We finally conduct a range of adaptive computations, using the {\it modified} estimators as well as the {\it original} residual based estimator and {\it a priori} graded meshes if applicable, for the three typical types of defects, namely micro-crack, anti-plane screw dislocation and multiple vacancies, where the later two are considered for the first time in the research of the adaptive a/c methods. The numerical results show that our {\it modified} error estimator with properly chosen parameters leads to the same convergence behavior as the  {\it original} error estimator while gain the efficiency by reducing the computational cost by one order with respect to the number of degrees of freedom. In addition, the advantage of the adaptive a/c methods is clearly demonstrated in the case of multiple vacancies where no {\it a priori} graded mesh may be derived at first. Such defect can also be considered as a simplification of the more complicated but realistic models which consider the interactions of multiple defects where the adaptive a/c may better reveal its power.

We constrain ourselves to the GRAC23 derived in \cite{PRE-ac.2dcorners} with the nearest neighbor interactions to clearly present the ideas and steps, though the analysis and constructions can be in principle extended to finite range interactions and to other a/c methods. We give a further remark for this in \S~\ref{sec:conclusion}.

\subsection{Outline}
\label{sec: outline}

%
In \S~\ref{sec:formulation} we introduce the crystalline defects and set up the atomistic, the continuum and the geometric reconstruction based atomistic/continuum (GRAC) coupling scheme considered in this work. 
In \S~\ref{sec:error} we first review the {\it a priori} and {\it a posteriori} error estimates for the a/c coupling scheme previously given in \cite{PRE-ac.2dcorners, wang2018posteriori, COLZ2013}. We then explain where the high computational cost previously noticed in \cite{liao2018posteriori} comes and present the theoretical analysis as well as the numerical justification of the relation between the modeling and the coarsening residual estimators. 
In \S~\ref{sec:eff_adapt_algo} we propose the modified error estimator and the corresponding adaptive algorithm based on the theoretical and numerical results in the previous section.
In \S~\ref{sec:numerics} we present the numerical results for the adaptive computations for the crystalline defects we consider and give a thorough discussion and explanation. We draw conclusions and discuss possible future works in \S~\ref{sec:conclusion}.

\subsection{Notations}
\label{sec:sub:not}

We use the symbol $\langle\cdot,\cdot\rangle$ to denote an abstract duality
pairing between a Banach space and its dual space. The symbol $|\cdot|$ normally
denotes the Euclidean or Frobenius norm, while $\|\cdot\|$ denotes an operator
norm.
%
%
For $E \in C^2(X)$, the first and second variations are denoted by
$\<\delta E(u), v\>$ and $\<\delta^2 E(u) v, w\>$ for $u,v,w\in X$.
%
The closed ball with radius $r>0$ and center $x$ is denoted by $B_r(x)$, or $B_r$ if the center is the origin.

%% file: formulation.tex

\section{Formulation of the Atomistic Model and its A/C Approximation}
\label{sec:formulation}

We present the atomistic fomulation and its a/c approximation of crystalline defects in an two-dimensional infinite lattice in this section. We keep the presentation as concise as possible since much of the detail can be found in various earlier works \cite{OrZh:2016, PRE-ac.2dcorners, wang2018posteriori, liao2018posteriori, COLZ2013}.




\subsection{Atomistic model}
\label{sec:formulation:atm}
%
%

\def\Rcore{B_{R^{\rm def}_i}}
\def\dfct{{\rm def}}

We denote a perfect single lattice which possesses no defects by $\Lhom:=\mathsf{A}\Z^2$ where $\mathsf{A} \in \R^{2 \times 2}$ can be considered as a macroscopic deformation gradient that often represents a uniform stretch or compression of a crystalline. To model local defects of a crystal lattice, we employ $\L\subset \R^2$ which satisfies the following assumption: 
\begin{assumption}\label{as:local}
$\exists~\Omega^\dfct :=\cup_{i=1}^{M} \Rcore(x_i)$,~such that~$\L\backslash \Omega^\dfct = \Lhom\backslash \Omega^\dfct$~and~$\L \cap \Omega^\dfct$ is finite.
\end{assumption}
The region $\Rcore(x_i)$ often contains localized defects with a center positioned at $x_i$ and is called a defect core. Such formulation can represent different types of defects. In particular, for point defects such as vacancies and micro-cracks we have $\L \subset \Lhom$, and for an anti-plane screw dislocation we have $\L \equiv \Lhom$. 

We denote a deformed configuration by a lattice function $y \in \Us$ where $\Us := \{v: \L\to \mathbb{R}^2 \}$ is the set of vector-valued lattice functions. The displacement $u\in \Us$ is then defined accordingly by $u(\ell) = y(\ell) - y_0(\ell) = y(\ell)-x(\ell) - u_0(\ell) = y(\ell) - \ell - u_0(\ell)$ for any $\ell\in\L$ if we define $x\in \Us$ to be the identity map. We consider $u_0 \in \Us$ as the {\it predictor}, which is usually a far-field crystalline environment, and $u\in \Us$ the {\it corrector} \cite{2013-defects,2021-defectexpansion}. We simply take $u_0=0$ for point defects and the derivation of $u_0$ for anti-plane screw dislocation is reviewed in  \ref{sec:appendix}. We refer to Figure \ref{fig:geom} for the illustration of the geometry of the three different types of defects we consider in this work, namely micro-crack, anti-plane screw dislocation and multiple vacancies.

%

\begin{figure}[!htb]
	\centering 
	\subfloat[Micro-crack]{
		\label{fig:geom_mcrack}
		\includegraphics[height=5cm]{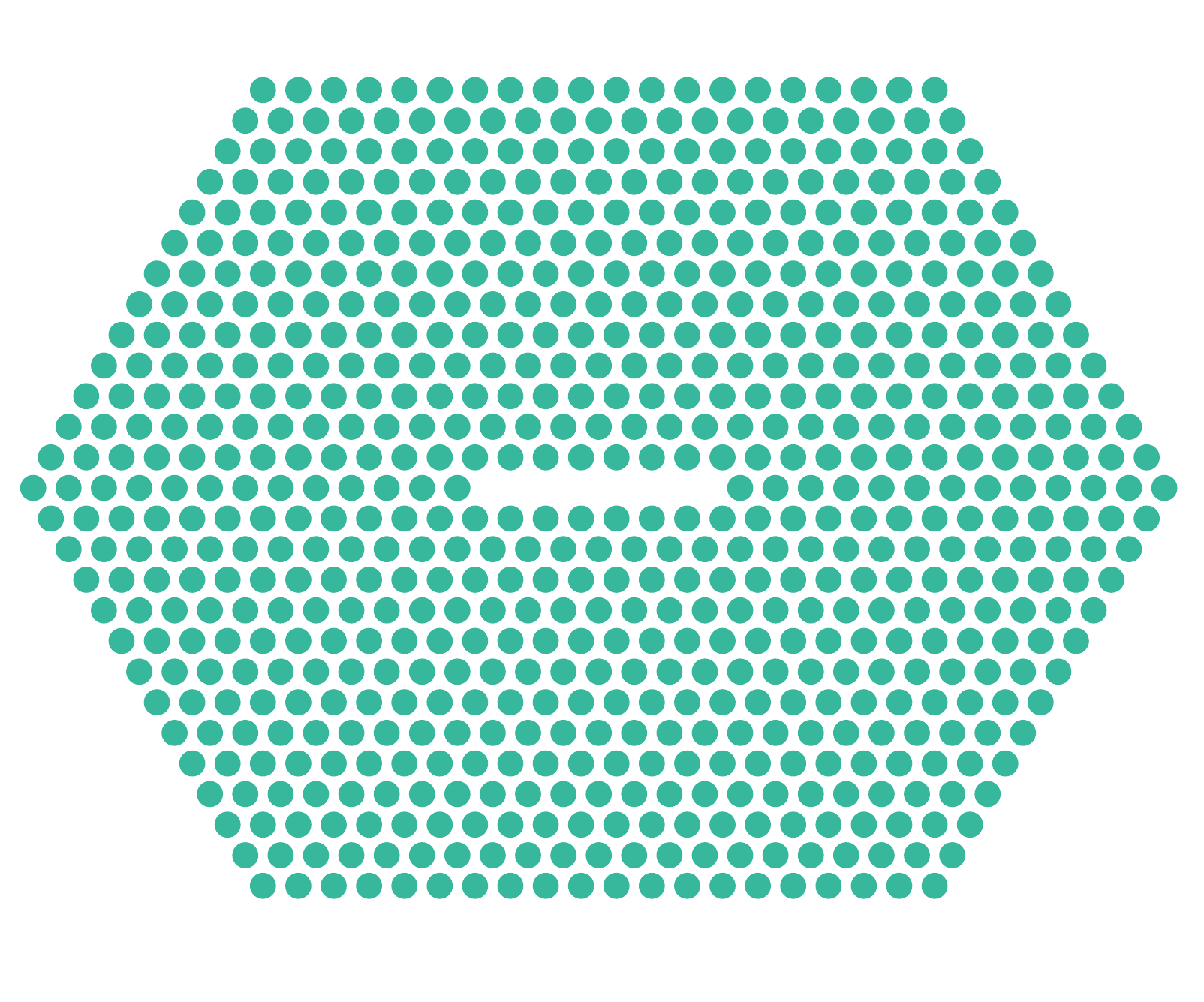}}
	\hspace{0.3cm} 
		\subfloat[Anti-plane screw dislocation (colors by {\it predictor} $u_0$)]{
		\label{fig:geom_screw} 
		\includegraphics[height=5cm]{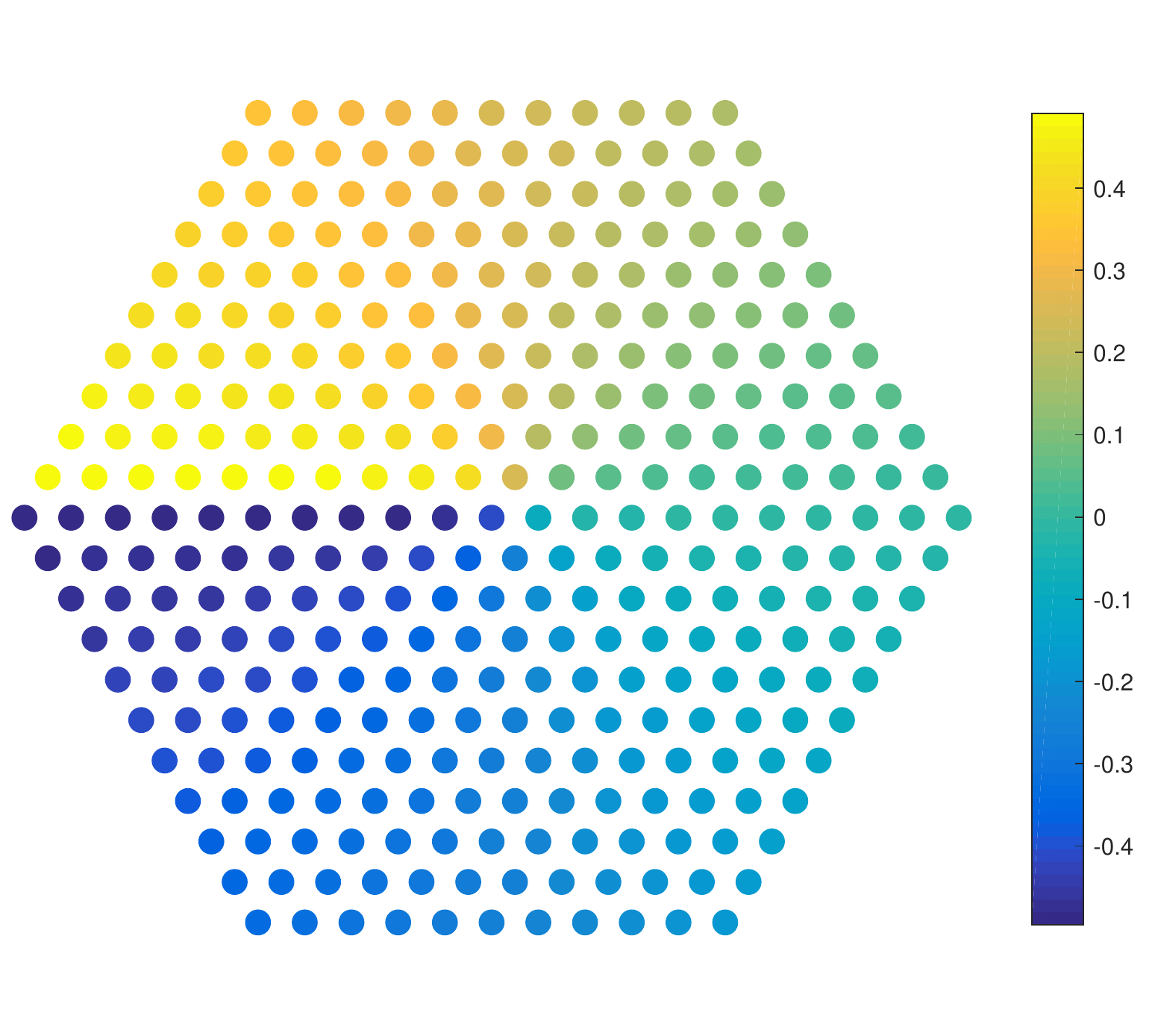}}
	\hspace{0.5cm} 
	\subfloat[Multiple separate vacancies]{
		\label{fig:geom_multivac}
		\includegraphics[height=5cm]{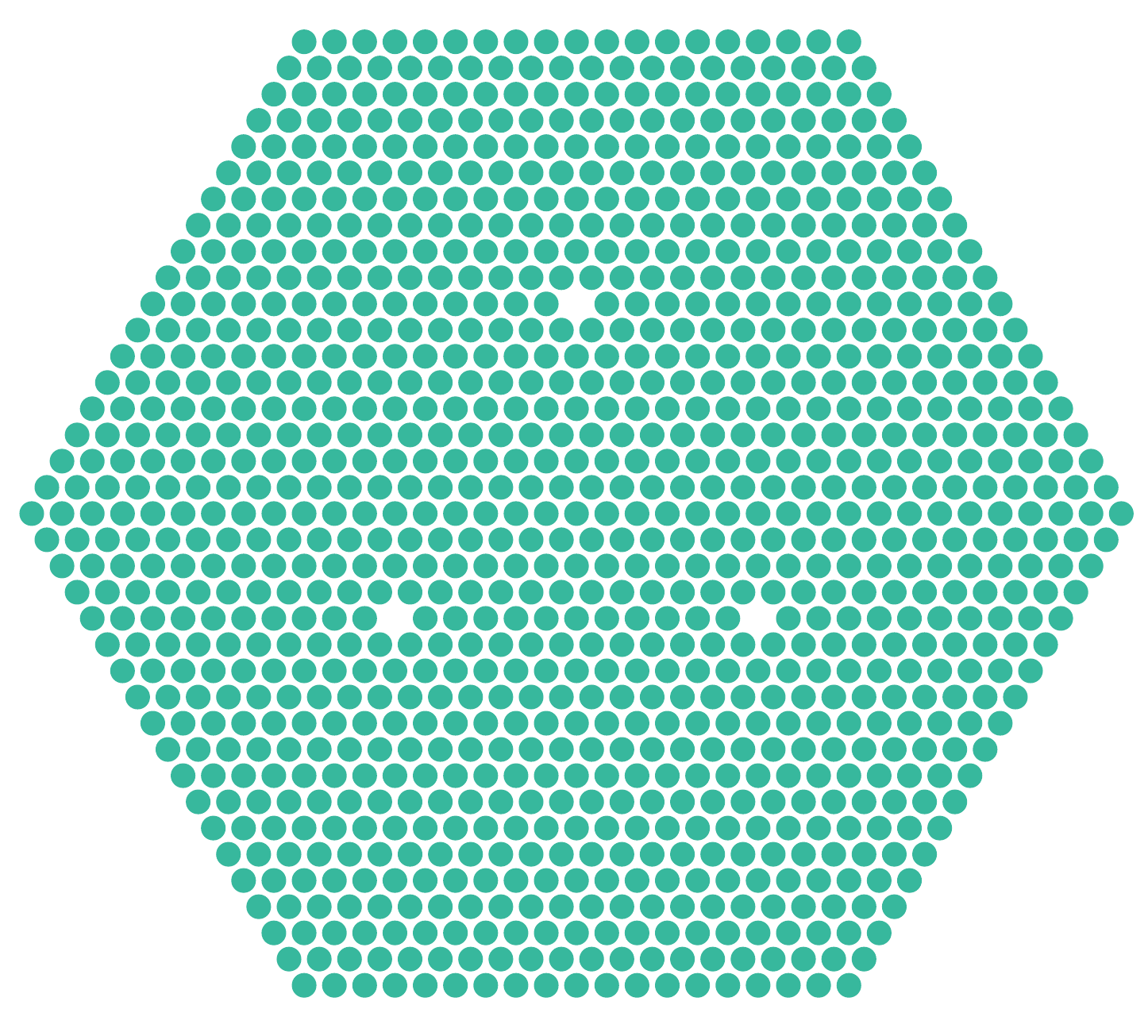}}		
	\caption{Illustration of three typical defects considered in this work.}
	\label{fig:geom}
\end{figure}

\begin{remark}
Assumption \ref{as:local} is presented in a slightly different form as that in various earlier works such as \cite{2013-defects, wang2018posteriori, 2021-defectexpansion, 2021-qmmm3}, where only a single round shaped defect core $B_{R^\dfct}$ is assumed. We decide to use the current formulation to include the case of multiple vacancies which has not been considered before. In a/c adaptive computations, the defect core is often contained in the atomistic region and the allowance of separating the defect core into a number of round shaped regions $\Rcore(x_i)$ facilitates the description of the initial steps of the adaptivity which are clearly visualized in Figure \ref{figs:evolution} in \S~\ref{sec:numerics:multiple vacancies}. However, we note that whether the defects are separated or ``centralized" have no fundamental influence on the analytical results we derive in the current work. Only the condition that $\L \cap \Omega^{\dfct} = \L \cap \cup_{i=1}^{M} \Rcore(x_i) $ (or $\L \cap B_{R^\dfct}$ in the case of a single round shaped defect core) plays the role which essentially means that the defect core can not appear at infinity.
\end{remark}

We define the interaction neighborhood for each $\ell\in \L$ by $\Nhd_{\ell} := \{ \ell' \in \L \setsep 0<|\ell'-\ell| \leq r_{\rm cut} \}$, where $r_{\rm cut}$ is called the cut-off radius and is always equal to $1$ in the current work since we only consider the nearest neighbor interaction. We also denote the interaction range $\Rg_\ell := \{\ell'-\ell \setsep \ell'\in \Nhd_\ell\}$ as the union of lattice vectors defined by the finite differences between the atom $\ell$ and the lattice points in its neighborhood $\Nhd_{\ell}$.

The finite difference stencil for each $\ell \in \Lambda$ is then defined by $Dv(\ell):= \{D_\rho v(\ell)\}_{\rho \in \Rg_\ell} :=\{v(\ell+\rho)-v(\ell)\}_{\rho \in \Rg_\ell}$. The lattice energy-norm (essentially a discrete $H^1$-semi-norm), which we use to measure the difference or error for lattice functions, is then defined by
\begin{equation}
  \| Dv \|_{\ell^2} := \bg( \sum_{\ell \in \L}
  \sum_{\rho \in \Rg_\ell} |D_\rho v(\ell)|^2
  \bg)^{1/2}, \ \ \forall v \in \Us.
\end{equation}



\def\Use{\Us^{1,2}}

To measure the local ``regularity" of a displacement function $u \in \Us$, it is convenient to assume the existence of a background mesh, that is, a {\it canonical} triangulation $\T_{\a}$ of $\R^2$ into triangles whose nodes are the reference lattice sites $\L$. We refer to \cite[Figure 1]{2013-defects} for the illustration of such triangulation in 2D. 
We then introduce a standard piecewise affine interpolation of $u$ with respect to $\T_{\a}$ by $I_{\a} u$. Notice that $\L$ is a lattice with some vacancies when we consider micro-crack and multiple vacancies. We can thus construct the piecewise interpolant with respect to $\Lhom$ by extending $u$ to the vacancy sites (see \cite[Appendix 1]{wang2018posteriori} for detail).



We identify $u=I_{\rm a}u$ if there is no confusion and denote the piecewise constant gradient $\nabla u := \nabla I_{\rm a} u$. We then introduce the discrete homogeneous Sobolev space
\begin{displaymath}
	\Use :=\{u\in \Us ~|~ \nabla u\in L^2\},
\end{displaymath}
with the semi-norm $|u|_{\Use} :=\|\nabla u\|_{L^2}$.

%




%

The atomistic system we consider is modeled by the site energy $V_\ell$ associated with the lattice points where $V_{\ell}:(\R^2)^{\Rg_\ell}\rightarrow\R$ represents the energy contribution by atom $\ell$. We assume that $V_\ell \in C^k((\R^2)^{\Rg_\ell}), k \geq 2$ and $V_\ell$ is \textit{homogeneous} outside the defect core $\Omega^{\dfct}$, namely, $V_\ell = V$ and $\Rg_\ell = \Rg$ for $\ell \in \Lambda \setminus \Omega^{\dfct}$. A further condition of point symmetry $\Rg = -\Rg$, and $V(\{-g_{-\rho}\}_{\rho\in\Rg}) = V(g)$ is also assumed to keep the homogeneity of the system outside the defect core \cite{2013-defects}.

We note that $V_\ell$ may vary according to $\ell$ inside each ${\Rcore}(x_i)$ so that certain defects can be modeled.  We also note that a great number of practical interatomic potentials satisfy the assumptions of smoothness and symmetry listed in the last paragraph including the widely used embedded atom method (EAM) \cite{Daw:1984a} and Finnis-Sinclair model \cite{FS:1984} so that it is not a serious restriction for realistic atomistic or material systems. We will use the empirical EAM potential throughout our numerical experiments, see \S~\ref{sec:numerics} for detail.

Since the energy of an infinite lattice is typically ill-defined, with a slight abuse of notation, we redefine the potential $V_\ell(y)$ as the energy difference $V_\ell(y) - V_\ell(x)$ for point defects, which is equivalent to assuming $V_{\ell}(x)=0$. As for anti-plane screw dislocation, we denote $V_{\ell}(y):=V(y)-V(y_0)$ (cf. \cite[Section 3.1]{2013-defects}). The energy functional of the atomistic system
\begin{equation}
\label{eqn:Ea}
  \Ea(y) = \sum_{\ell \in \L} V_\ell(y)
\end{equation}
is then a meaningful object. By the assumptions of smoothness and symmetry on $V_\ell$, $\Ea(y)$ is well-defined for $y-y^{B}  \in \Use$ with certain Fr\'{e}chet differentiability, where $y^B(x)=Bx+u_0(x)$. We refer to \cite{2013-defects} for a detailed discussion.


The atomistic problem we would like to solve is to find a \textit{strongly stable}
equilibrium $y$, such that, given a macroscopic applied strain $B\in\R^{2\times 2}$, 
\begin{equation}
  \label{eq:min}
  y \in \arg\min \b\{ \Ea(y)~\bsetsep~y-y^B\in \Use \b\}.
\end{equation}
The \textit{strong stability} on $y$ means that there exists a constant $c_0 > 0$ such that
\begin{displaymath}
\< \ddel \Ea(y) v, v \> \geq c_0 \| \nabla v \|_{L^2}^2, \quad \forall v \in \Us^{1,2}.
\end{displaymath}

\subsection{GRAC coupling method}
\label{sec:formulation:ac}

\def\Th{\mathcal{T}_h}
\def\Nh{\mathcal{N}_h}
\def\Ush{\Us_h}
\def\Ra{R^\a}
\def\Rb{R^{\rm b}}
\def\Rc{R^\c}
\def\Eb{\E^{\rm b}}
\def\dof{{\rm DOF}}
\def\Omh{\Omega_h}
\def\Thr{{\T_{h,R}}}
\def\vor{\rm vor}
\def\Uhr{\Us_{h,R}}

%
The atomistic problem \eqref{eq:min} is not computable in practice since it is defined on an infinite domain and it considers every atom as a degree of freedom. The former problem is easily circumvented by restricting the computational domain to be $\Omega \subset \R^2$ which is a simply connected polygonal and we assume that $\exists R \in \R$ such that $B_{R} \subset \Omega \subset B_{c_0 R}$ for some $c_0 >1$ where $R$ is the radius of $\Omega$.

To deal with the later issue, namely the overwhelming number of degrees of freedom, we first define a continuum approximation by the Cauchy-Born rule which is both compatible with \eqref{eqn:Ea} and related to the strain. The Cauchy-Born energy density functional $W : \R^{2 \times 2} \to \R$ is simply given by
\begin{displaymath}
  W(\mF) := \det (\mA^{-1}) \cdot V(\mF \mathcal{R}),
\end{displaymath}
where $V$ is the homogeneous site energy potential on $\Lhom$ \cite{E:2007a, OrtnerTheil2012}. An atomistic-to-continuum (a/c) coupling method then essentially couples the atomistic model with its Cauchy-Born approximation together. To formulate our coupling scheme, we first decompose the computational domain into two parts: $\Omega = \Omega^\a \bigcup \Omega^\c$ where the defect core $\Omega^{\dfct}$ is contained in $\Omega^\a$. Next, we decompose the set of atoms inside the atomistic region into a core atomistic set $\L^\a$ and an interface set $\L^\i$ (usually a few ``layers'' of atoms surrounding $\L^\a$ for each $\Rcore(x_i)$) such that $\L^{\a,\i} := \L \bigcap \Omega^\a = \L^\a \bigcup \L^\i$ and $\L\cap\Omega^{\dfct}$ is contained in $\L^\a$.


To reduce the number of degrees of freedom, we introduce the finite element partition of the computational domain $\Th$ so that $\Th = \T^{\c}_{h} \bigcup \T^{\a}_{h}$ where $\T^\a_{h}$ is the {\it canonical} triangulation induced by $\L^{\a,\i}$ and $\T^{\c}_{h}$ be a shape-regular simplicial partition (triangles for 2D) of the continuum region $\Omega^\c$. We note that $\T^\a_{h}$ may contain "holes" due to the existence of defects (see Figure \ref{figs:plotMesh} for an illustration). We define $\Nh$ be the set of nodes in $\Th$ and $\mathcal{F}_h^{\c}$ be the set of edges in $\Th^{\c}$.

Assuming $P_1$ finite element are used on $\Th$, the general formulation of an energy-based atomistic-to-continuum coupling energy functional can be written as 
\begin{align}
  \label{eq:generic_ac_energy}
  \Eh(y_{h}) := & \sum_{ \ell \in \L^\a}  V_\ell(y_h)  + \sum_{\ell \in \L^\i} \omega_\ell V^\i_\ell(y_h)  + \sum_{T \in \Th} \omega_T  W(\D y_h|_T),   
\end{align}
where $V_\ell^\i$ is a modified interface site potential and $\omega_\ell$ and $\omega_T$ are certain coefficients which are usually called effective volumes of lattice sites or elements (see \cite[Section 2.2]{PRE-ac.2dcorners} for a detailed discussion). We note that $\omega_T=0$ for $T\in\T_h^{\a}\setminus\T_h^{\rm i}$ with $\T_h^{\rm i}:=\{T\in\T_h: \L^{\rm i}\cap T \neq \emptyset\}$.

The interface potential in \eqref{eq:generic_ac_energy} is the key for constructing a consistent a/c coupling method. In the current work, we restrict ourselves to the geometric reconstruction based consistent a/c (GRAC) coupling method which is first proposed in \cite{PRE-ac.2dcorners} and then further developed in \cite{COLZ2013}. The essential idea for GRAC method is to introduce the geometric reconstruction parameters $C_{\ell;\rho,\vsig}$ so that for each $\ell \in\L^\i, \rho, \vsig \in \Rg_\ell$ the interface potential is redefined by 
\begin{equation}
  \label{eq:defn_Phi_int}
  V_\ell^\i(y) := V \B( \b( {\textstyle \sum_{\vsig \in
      \Rg_\ell} C_{\ell;\rho,\vsig} D_\vsig y(\ell) } \b)_{\rho \in
    \Rg_\ell} \B).
\end{equation}
The geometric reconstruction parameters are then determined by solving the following equations:
\begin{equation}
  \label{eq:energy_pt}
  V_\ell^\i(y^{\mF}) = V(y^{\mF}) \qquad \quad \forall\ell \in \L^\i, \quad \mF \in \R^{d
    \times d},
\end{equation}
and 
\begin{equation}
  \label{eq:force_pt}
  \< \del \Eh(y^{\mF}), v \> = 0 \qquad \forall v \in \UsT,
  \quad \mF \in \R^{d \times d},
\end{equation}
which are termed as the energy patch test consistency and the force patch test consistency \cite{Or:2011a}. According to \cite[Proposition 3.7]{PRE-ac.2dcorners}, under the setting of the current work, for $\ell\in\L^{\i}$, the parameters $C_{\ell;\rho, \zeta}$ in \eqref{eq:defn_Phi_int} are determined by: (a) $C_{\ell, \rho, \vsig} = 0$ for $|(\rho - \vsig)~\textrm{mod}~6| > 1$; (b) $C_{\ell, \rho, \rho-1} = C_{\ell, \rho, \rho+1} = 1 - C_{\ell, \rho, \rho}$; (c) $C_{\ell, \rho, \rho} = C_{\ell+\rho, -\rho, -\rho} = 1$ for $\ell+\rho\in\L^{\i}$; (d) $C_{\ell, \rho, \rho}=1$ for $\ell+\rho \in \L^{\a}$; (e) $C_{\ell, \rho, \rho}=2/3$ for $\ell+\rho\in \L^{\c}$. The coefficient $2/3$ given in (e) introduces the name of the specific GRAC method we consider in the rest of this paper, which is the GRAC23 method. However, we need to note that these coefficients are not unique and how they are optimally determined are discussed in depth in \cite{COLZ2013}.

Let $\Omega_{h} = \bigcup_{T\in \Th} T$. We define the space of \textit{coarse-grained} displacements to be
\begin{align*}
  \Us_{h} := \b\{ u_h : \Omega_{h} \to \R^2 \bsetsep &
  \text{ $u_h$ is continuous and p.w. affine w.r.t. $\T_{h}$, } \\[-1mm]
  & \text{ $u_h = 0$ on $\partial \Omega_h$ } \b\}.
\end{align*}
The a/c coupling problem we would like to solve is to find  
\begin{equation}
  \label{eq:min_ac}
  y_{h} \in \arg\min \b\{ \Eh(y_h)~\bsetsep~ y_h - y^B  \in
  \Us_{h} \b\}.
\end{equation}

\def\Ta{\T_\a}
\def\Th{\T_{h}}
\def\sh{\sigma^{\ac}}

\subsubsection{The formulation of the stress tensors}
\label{sec:formulation:stress}
The stress tensor, which is in principle defined by the first variation of the energy, plays the key role in both {\it a priori} and {\it a posteriori} error analysis of the a/c coupling methods. For any $y -y^B \in \Use$, and $y_h-y^B\in \Ush$, we define the atomistic stress tensor $\sa(y; \cdot)\in \PO(\Ta)^{2\times 2}$, the continuum stress tensor $\sigc(y_h; \cdot) \in \PO(\Th)^{2\times 2}$ and the a/c stress tensor $\sh(y_h; \cdot) \in \PO(\Th)^{2\times 2}$ respectively by the following identities
\begin{align}
  \<\del\Ea(y),v\>&=\sum_{T\in\Ta}|T|\sa(y;T) \nabla_T v, \qquad  \forall v\in\Use, 
  \label{eq:atomstress}\\
  \<\del\Ec(y_h),v_h\>&=\sum_{T\in\Th}|T|\sigc(y_h;T) \nabla_T v_h, \qquad \forall v_h\in\Ush,
  \label{eq:contstress}\\
    \<\del\Eh(y_h),v_h\>&=\sum_{T\in\Th}|T|\sh(y_h;T) \nabla_T v_h, \qquad \forall v_h\in\Ush,
  \label{eq:acstress}
\end{align}
where $\Ta$ is the {\it canonical} triangulation induced by the reference lattice $\L$ and~$\Ec(y_h):=\sum\limits_{T\in\Th}|T|W(\nabla_T y_h)$. Since we only consider nearest neighbor interactions, we can choose the following stress tensors from the first variations~\eqref{eq:atomstress}-\eqref{eq:acstress},
\begin{align}
    \label{eq:defn_Sa}
    \sa(y; T) :=~& \frac{1}{\det \mA} \sum_{b=(\ell, \ell+\rho)\in \partial T} \partial_\rho V_\ell \otimes
    a_\rho, \\
    \label{eq:defn_Sc}
    \sigc(y_h; T) :=~& \pp W(\Dc{T} y_h)=\frac{1}{\det \mA} \sum_{j = 1}^6 \partial_j V(\nabla_T y_h) \otimes a_j, \\
    \label{eq:defn_Sh}
    \sh(y_h; T) :=~\frac{1}{\det\mA}&\sum_{b=(\ell, \ell+\rho)\in \partial T} \partial_\rho V_\ell(\Ia y_h) \otimes
    a_\rho + \omega_T \sigc(y_h; T).
\end{align}
We note that the stress tensors defined in \eqref{eq:atomstress}-\eqref{eq:acstress} are not unique. Equation \eqref{eq:atomstress}-\eqref{eq:acstress} may hold with any of the stress tensors differ from \eqref{eq:defn_Sa}-\eqref{eq:defn_Sh} with a piecewise constant tensor field $\sigma\in \PO(\T)^{2\times 2}$ that satisfies the \textit{divergence free} condition
\begin{equation}
\label{eq: divergence_free_condition}
	\sum_{T\in\T}|T|\sigma(T):\Dc{T} v\equiv 0, \qquad \forall v\in ({\rm P}_1(\T))^2.
\end{equation}
We postpone the detail of the construction and optimization of such divergence free field to \S~\ref{sec:Interfacial-stress} where a local stress tensor correction is applied to the coupling stress tensor $\sigma^{\rm ac}$.

%% file: errorAnal.tex
\section{The Analysis of the Residual Based A Posteriori Error Estimates}
\label{sec:error}

In this section, we give the {\it a posteriori} error estimates for the GRAC method introduced in \S~\ref{sec:formulation:ac}. We first present some useful {\it a priori} results in \S~\ref{sec:sub:aprior} that motivates the separation of the {\it a posteriori} error estimator. We then give a brief review of the {\it original} residual based {\it a posteriori} error estimates in \S~\ref{sec:sub:aprior} which is derived and analyzed in \cite{wang2018posteriori}. The error estimates provide a solid foundation of the current work but the problem of high computational cost arise in computing the modeling part of the {\it a posteriori} error estimator. We finally concentrate on the analysis and numerical verification of the relation between the modeling and the coarsening parts of the estimator whose quotient in fact possesses a certain ratio under suitable assumptions. Such ratio relation serves as the theoretical basis for deriving modified {\it a posteriori} error estimators with much better computational efficiency that are discussed in \S~\ref{sec:eff_adapt_algo} and \S~\ref{sec:numerics}. We note that the content in \S~\ref{sec:sub:aprior} and \S~\ref{sec:error:aposterior} can be found in detail in earlier works \cite{wang2018posteriori, liao2018posteriori} and we include them only for completeness and refer to related literature where necessary.


%

\subsection{A Priori Error Estimates}
\label{sec:sub:aprior}

If $y$ is a strongly stable solution of \eqref{eq:min} and is sufficiently smooth in the continuum region, there exists a strongly stable solution $y_{h}$ to \eqref{eq:min_ac} and a constant $C^{a-priori}$ for GRAC method such that \cite{2013-defects, LuOr:acta},
\begin{align}\label{eq:ap}
	\|\nabla u_{h} - \nabla u\|_{L^2} \leq C^{a-priori} \big( \|D^2u\|_{\ell^2(\mathcal{I}_{\rm ext})} + \|D^3u\|_{\ell^2(\L\cap\Omega^\c)} + \|D^2u\|^2_{\ell^4(\L\cap\Omega^\c)} \nonumber \\
	+ \|hD^2 u\|_{\ell^2(\L\cap\Omega^\c)}+\|Du\|_{\ell^2(\L \setminus B_{R/2})}\big),
\end{align}
where $u = y-y^B, u_{h} = y_{h} - y^B$ and $\mathcal{I}_{\rm ext}$ is the extended interface region whose definition is given in \cite[\S~5]{PRE-ac.2dcorners}. The right hand side of \eqref{eq:ap} can be essentially split into three parts and we give a discussion here as 
it helps to understand the correspondence between the {\it a priori} error estimates and the {\it a posteriori} error estimates.

The first part is the \textit{modeling error} defined by $\zeta_\mo:=\|D^2u\|_{\ell^2(\mathcal{I}_{\rm ext})} + \|D^3u\|_{\ell^2(\L\cap\Omega^\c)} + \|D^2u\|^2_{\ell^4(\L\cap\Omega^\c)}$. The error is due to the difference between the atomistic model and the continuum model which lies in both the discrepancy of the nonlocality and the locality of the atomistic and the coupling energy and the lattice and mesh on which we solve the problems. This is the fundamentally different from the discretization of a PDE using certain numerical schemes, see \cite{PRE-ac.2dcorners, wang2018posteriori} for further discussion.  

The second part is the \textit{coarsening error} denoted by $\zeta_\cg:=\|hD^2 u\|_{\ell^2(\L\cap\Omega^\c)}$. This is caused by the application of the finite element method in the continuum region and is a reminiscence of the discretization error of the numerical solution to an elliptic differential equation.

The third part is the \textit{truncation error} denoted by $\zeta_\trc:=\|Du\|_{\ell^2(\L\setminus B_{R/2})}$, which arises from the finite size of our computational domain. We note that the region $B_{R/2}$ comes from the construction of the truncation operator, see \cite[Lemma 7.3]{2013-defects} for a better detail.

\subsection{Residual based a posterior error estimates}
\label{sec:error:aposterior}
%


We derive the {\it a posteriori} error estimates based on the residual for the a/c method in this section. To avoid the unnecessary complexity of the presentation while keeping the main ideas and steps, we restrict our discussion to the case of the GRAC23 method with nearest neighbor multibody interactions. We refer to \cite{liao2018posteriori} for the extension to the finite range interaction. We also note the derivation is concise as the detail can be found in \cite{wang2018posteriori}.


\def\Pdef{\mathscr{P}^{\rm def}}
\def\Usc{\Us^c}
\def\Use{\Us^{1,2}}
\def\Usrh{\Us_R^h}


The first variation of the atomistic variational problem \eqref{eq:min} and that of the a/c coupling problem  \eqref{eq:min_ac}  are to find $y-y^B \in \Use$ and $y_h - y^B \in \Us_{h}$ such that 
\begin{equation}
\label{eqn:firstvariationea}
	\<\del\Ea(y), v\> = 0, \quad \forall v \in \Use, \text{ and } \<\del\Eh(y_h), v_h\> = 0, \quad \forall v_h\in \Us_{h},
\end{equation}
respectively. Since we limit ourselves in a finite computational domain $\Omega$, we choose the same truncation operator $T_R: \Us^{1,2} \to \Us_R$ where $\Us_R := \{u \in \Use~|~u(x) = 0,~\forall x\in \L \backslash \Omega\}$. 




The residual operator $\mR$ is then defined as an operator on $\Use$ and $\mR[v]$ is given by 
\begin{equation}
\mR[v] = \< \del\Ea(\Ia y_h), v\>, \quad \forall v\in \Use.
\end{equation}
Let $v_R = T_R v$, and $v_h = \Cs_h T_R v: \Use\to \Us_{h} $, where $\Cs_h: \Us_R\to \Us_{h}$ is the modified Cl\'{e}ment operator \cite{wang2018posteriori, clement:1975, Verf:1999}. By the second equation in \eqref{eqn:firstvariationea}, we separate the residual $\mR[v]$ into three groups
\begin{align}
\label{eq: residual_separation}
	\mR[v] = \< \del\Ea(\Ia y_h), v\>  = & \< \del\Ea(\Ia y_h), v\> - \<\del\Eh(y_h), v_h\> \notag \\
	 = & \big\{ \< \del\Ea(\Ia y_h), v_R\> - [ \del\Eh(y_h),  v_R] \big\} + \big \{  [ \del\Eh(y_h), v_R] - \<\del\Eh(y_h), v_h\> \big\} \notag \\
	 &+ \big\{\< \del\Ea(\Ia y_h), v\> - \< \del\Ea(\Ia y_h), v_R\> \big\}
	=: \mR_\mo + \mR_\cg + \mR_\trc,
\end{align}
where the operator $[\cdot, \cdot]$ is defined by
\begin{align}
	[ \del\Eh(y_h), v_R]:= \sum_{T\in\Th}\int_T\sh(y_h, T) \nabla v_R \dx 
				       = & \sum_{T\in\Th}\sh(y_h, T) (\sum_{T'\in\Ta, T'\bigcap T \neq \emptyset}|T \bigcap T'|\nabla v_R) \notag \\
				       = &\sum_{T\in \Ta}|T| \big(\sum_{T'\in \Th, T'\bigcap T\neq \emptyset}\frac{|T'\bigcap T|}{|T|}\sh(y_h, T')\big)\nabla v_R, \label{eq:ehvr}
\end{align}
which is different from $\<\cdot, \cdot\>$ since $v_R\notin \Us_h$. We essentially separate the residual into the modeling part $\mR_\mo$, the the coarsening part $\mR_\cg$ and the truncation part $\mR_\tr$.




\def\Tr{\rm Tr}


The estimate of the truncation residual $\mR_\trc$ is the easiest among the three and is simply given by 
\begin{equation}
\label{eqn:etatrc}
	|\mR_\trc| \leq C^{\trc}\|\sa(\Ia u_{h})-\sigma^0\|_{L^2(\Omega \setminus B_{R/2})}\|\nabla v \|_{L^2} =: \eta_\trc(y_{h})  \|\nabla v \|_{L^2},
\end{equation}
where $C^{\trc}$ is the constant appeared when estimating the difference of $Dv_R$ and $Dv$ in terms of $Dv$, and  we can specify the divergence-free stress tensor $\sigma^0 = \pp W(\nabla y^B)$ (cf. \cite[\S 3.1.1]{wang2018posteriori}).

\def\sjump{\llbracket\sh\rrbracket_f}
\def\dx {\rm dx}
\def\Fh {\mathcal{F}_h}
\def\Fhi{\Fh^{\rm c}\bigcap {\rm int}(\Omega)}

The coarsening residual $\mR_\cg$ is estimated by 
\begin{equation}
\label{eqn:etacg}
	|\mR_\cg| \leq \sqrt{3}C^{\rm tr}C_{\Th}(\sum_{f\in\Fhi} (h_f\sjump)^2)^\frac12 \|\nabla v\|_{L^2(\Omega)}=: \eta_\cg(y_h) \|\nabla v\|_{L^2(\Omega)},
\end{equation}
where $\Fhi$ is the set of interior edges, $\llbracket\sh\rrbracket_f:=\sh(y_h, \Tp_f)\np+\sh(y_h, \Tm_f)\nm$ denotes the jump of $\sh$ across each interior edge $f$ and $C_{\Th}$ is the constant for the modified Cl\'{e}ment operator that only depends on the shape regularity of the mesh. We refer to \cite[\S 3.1.3]{wang2018posteriori} for a thorough discussion.


For the analysis of the modeling error $\mR_\mo$, the formulation of $\< \del\Ea(\Ia y_h), v_R\>$ and  $[ \del\Eh(y_h), v_R]$ given in \eqref{eq:acstress} and \eqref{eq:ehvr} lead to \cite[\S 3.1.2]{wang2018posteriori},
\begin{equation}
\label{eqn:etamo}
	|\mR_\mo| \leq C^{\trc}\big\{\sum_{T'\in\Ta}|T|\big[\sa(\Ia y_h, T') - \sum_{T\in \T_h, T\bigcap T'\neq \emptyset}\frac{|T\bigcap T'|}{|T'|}\sh(y_h, T)\big]^2\big\}^{\frac12}\|\nabla v\|_{L^2} =: \eta_\mo(y_h) \|\nabla v\|_{L^2}.
\end{equation}

\begin{remark}
\label{rmk: complex geometry}
We have to note that though the estimate of the modeling residual $\eta_\mo$ is probably the most straight forward one among the three, it consumes the most computational cost in the adaptive computation. It is easily observe from \eqref{eqn:etamo} that to compute $\eta_\mo$ we need to find $|T'\bigcap T|$ for every $T\in\Ta$ and $T'\in \T_h$. Thus the complex geometry, which is illustrated soon in Figure \ref{figs:TcapTh}, enters into the computation when we assign the global error estimator into local contributions in each adaptive step. We give an estimate of the computational cost of the modeling residual in \S~\ref{sec:local} and from then on, we will discuss how to deal with this inefficiency that eventually reduces the computational cost by one order with respect to the number of degrees of freedom.
%
%
%
\end{remark}

Combining the residual estimates in \eqref{eqn:etatrc}, \eqref{eqn:etacg} and \eqref{eqn:etamo}, the stability results in \cite[Section 3.2]{wang2018posteriori} and a quantitative version of the inverse function theorem \cite[Lemma B.1]{LuOr:acta}, we can obtain the following theorem for the {\it rigorous global a posteriori existence and error estimate} \cite[Theorem 3.7]{wang2018posteriori}. 



\begin{theorem}
\label{thm:h1error}
Let $y_h$ be the a/c solution defined in \eqref{eqn:firstvariationea}, $\eta(y_h) = \eta_{\trc}(y_h)+\eta_\mo(y_h)+\eta_\cg(y_h)$ be the total residual estimator, $ \gamma(y_h)$ be the stability constant defined in \cite[\S~3.2.4]{wang2018posteriori}, and $L_1$ be the Lipschitz constant of $\ddel \Ea$ whose definition can be found in \cite[Lemma 3.9]{LuOr:acta}. Under the assumption that $\gamma(y_h)>0$ and $2L_1\eta(y_h)< \gamma(y_h)^2$, there exists a unique $y$ satisfying $y - y^B \in \Use$ which solves the atomistic variational problem \eqref{eqn:firstvariationea}, and satisfies the following {\it a posteriori} error bound, 
\begin{equation}
	\|\nabla \Ia y_h - \nabla y\|_{L^2} \equiv \|\nabla \Ia u_h - \nabla u\|_{L^2}\leq 2\frac{\eta(y_h)}{\gamma(y_h)}.
\end{equation}

\end{theorem}

\subsection{Local contribution and analysis of the modeling and the coarsening residual estimators}
\label{sec:local}

\def\sjump{\llbracket\s^{\ac}\rrbracket}

The {\it a posteriori} error estimates we just proposed is global. In order to design and implement the adaptive algorithm, we need to assign the global residual estimator $\eta(y_h)$ into elementwise local contributions except for the truncation part. The truncation residual estimator is often computed directly in practice since it is relatively small if the computational domain $\Omega$ is sufficiently large. We then check whether it exceeds the threshold during the main adaptive process. See Algorithm~\ref{alg:size} in \S~\ref{sec: local_estimator_and_adaptiv_algorithm} for detail.



Following standard adaptive finite element procedure, the local coarsening residual estimator for an element $T\in\T_h$ can be defined as 
\begin{equation}\label{eq:localetaC}
\eta_{\cg}(y_h;T) := \Bigg( \sum_{f\in \Fh\bigcap T\in\T_h}\frac12(h_f\sjump_f)^2 \Bigg)^{\frac{1}{2}}.
\end{equation}

The local modeling residual estimator for an element $T\in\T_h$ is defined as the sum of the differences between stress tenors of $T$ and those of $T' \in\Ta$ satisfying $T\bigcap T'\neq \emptyset$ so that
\begin{equation}
  \eta_{\mo}(y_h;T) = \Bigg(\sum_{T'\in \Ta, T'\bigcap T\neq \emptyset} \eta_{\mo}(y_h; T', T)\Bigg)^{\frac{1}{2}}, \quad \forall T\in \T_h,
  \label{eq:localetam}
\end{equation}  
where
\begin{displaymath}
\eta_{\mo}(y_h;T', T):= |T'\bigcap T|\Bigg(\sa(\Ia y_h, T') - \sum_{T\in \T_h, T'\bigcap T\neq \emptyset}\frac{|T\bigcap T'|}{|T'|}(\sh(y_h, T))\Bigg)^2. 
\end{displaymath}
We note that the global modeling residual estimator $\eta_{\mo}(y_h)$ defined in \eqref{eqn:etamo} is summed over $\T_{\rm a}$ for the simplicity of presentation while the elementwise modeling residual estimator is defined on each $T \in \T_h$ for the purpose of adaptive mesh refinements, and the summation of $\eta_{\mo}(T)$ over $T \in \T_h$ and $\eta_{\mo}(y_h)$ are essentially equal up to a constant $C^{\rm tr}$. We omit the constants $C^{\rm tr}$ and $C_{\T_h}$ here only for the sake of presentation and will determine the values of these constants empirically in our adaptive computations (cf. \S~\ref{sec: local_estimator_and_adaptiv_algorithm}).

As we have already pointed out in Remark \ref{rmk: complex geometry}, evaluating $\eta_{\mo}(T')$ requires computing $|T\cap T'|$ for every $T \in \Ta$ and $T'\in\T_h$ where finding the geometric relationship between $T$ and $T'$ is essential. See Figure \ref{figs:TcapTh} for an illustration of the mismatch of $T \in \Ta$ and $T'\in\T_h$. We note that the computational cost of evaluating $\eta_{\mo}$ is proportional to the number of atoms in the continuum region since in our implementation we loop over all the elements in $\T_{\a}$ lying in the continuum region to record the geometric relationship. The number of atoms in the continuum region is about the square of the number of elements in $\T_h$. Moreover, the ratio of the number of atoms in the atomistic region and the number of elements in $\T_h$ remains approximately a constant under certain assumptions (see \cite[Section 8.1]{LuOr:acta} for detail). Hence, if $N$ is the total number of degrees of freedom of the a/c coupling problem \eqref{eq:min_ac}, then the computational cost of evaluating $\eta_{\mo}(T')$ can be roughly estimated by $O(N^2)$.  In contrast, the cost of evaluating the coarsening residual $\eta_{\cg}(T)$ is proportional to the number of edges $f \in \F^{\rm c}_{h}$ of the elements in the continuum region, which makes it only of $O(N)$. 
We give a numerical verification of the computational cost in \S~\ref{sec:numerics} (cf. Figure \ref{fig:time_mcrack} for micro-crack, Figure \ref{fig:time_screw} for anti-plane screw dislocation and Figure \ref{fig:time_multivac} for multiple vacancies).


\begin{figure}[htb]
\begin{center}
	\includegraphics[scale=0.4]{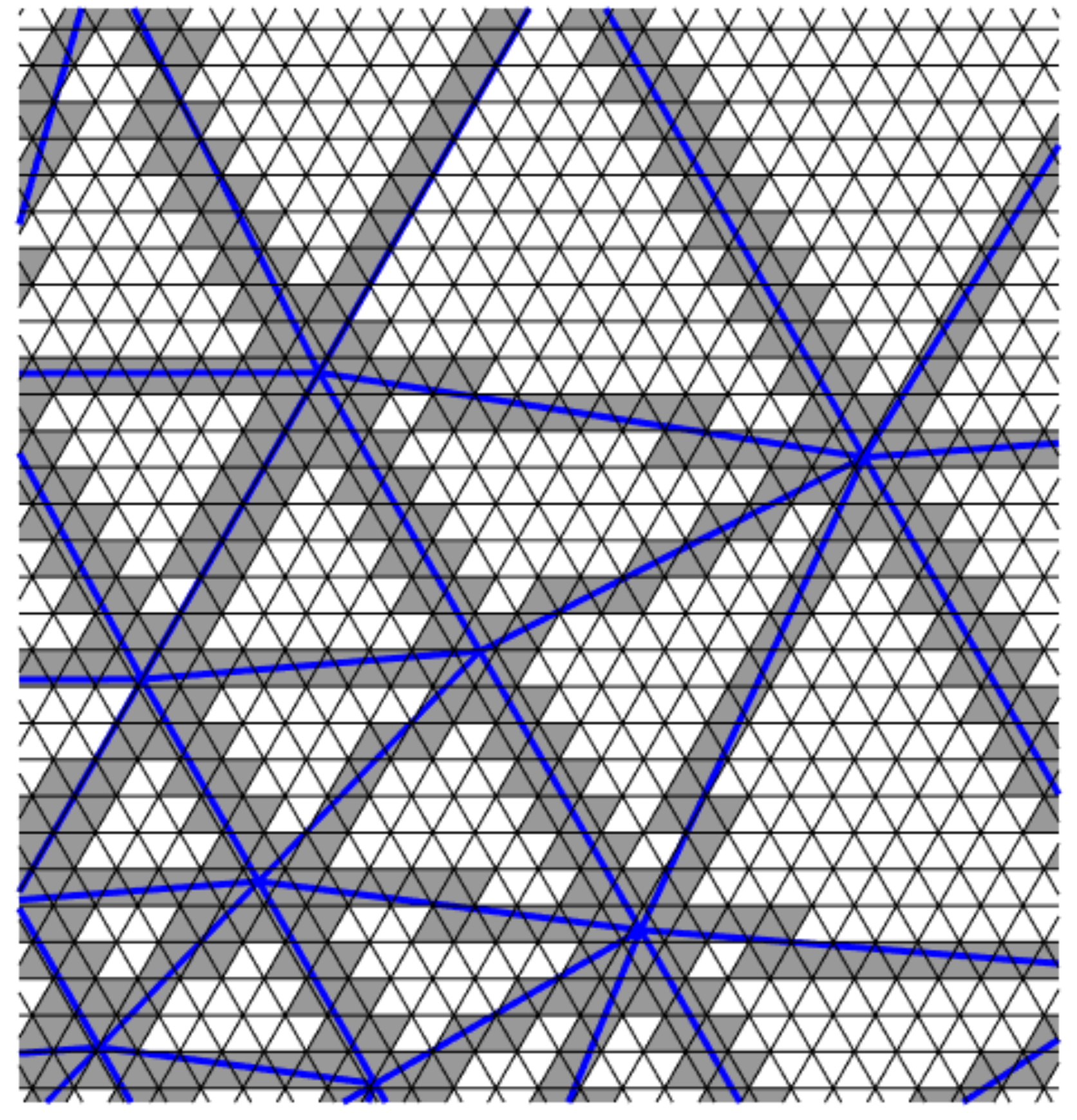}
	\caption{Illustration of $|T'\cap T|$. The edges of $T\in\T_h$ are marked with long blue solid lines whereas the elements $T' \in \T_\a$ at the boundaries of the elements $T\in\T_h$ are marked in grey.}
	\label{figs:TcapTh}
\end{center}
\end{figure}

To avoid the exceedingly high cost of computation, we apply an approximation of the elementwise residual estimator $\eta_{\mo}(y_h;T)$ in our adaptive computations. The idea of the approximation comes from a simple observation that the elementwise modeling residual estimator $\eta_{\mo}(y_h;T)$ should somehow be ``negligible" compared with the elementwise coarsening residual estimator $\eta_{\cg}(y_h;T)$. Such observation may be easily verified if we take the spring potential (essentially a linearized model) as an example. With the assumption that $u$ is relative smooth in the regions that are far-away from the defect core, we can directly estimate $\eta_{\cg}(y_h;T) \approx O(h_{T}^2)$ and $\eta_{\mo}(y_h;T) \approx O(h_{T})$ by some algebraic manipulations according to the definitions \eqref{eq:localetaC} and \eqref{eq:localetam}. Since $h_T \gg 1$ for any $T \in \T_{h}$ that is sufficiently far away from the defect core, we have a good reason to believe that $\eta_{\mo}(y_h;T)$ should be a quantity of higher order compared  with $\eta_{\cg}(y_h;T)$. For the rest of this section, we make this argument rigorous and numerically verify its validity.


We start with the {\it a priori} error estimates \eqref{eq:ap}. We can further split the modeling error by 
\begin{align}\label{eq:zeta}
\zeta_{\mo} = \|D^2u\|_{\ell^2(\mathcal{I}^{\rm ext})} + \|D^3u\|_{\ell^{2}(\L\cap\Omega^{\c})} + \|D^2u\|^2_{\ell^{4}(\L\cap\Omega^{\c})} =: \zeta_\mo^\i(y) + \zeta_\mo^\c(y),
\end{align}
where $\zeta_\mo^\i(y)$ and $\zeta_\mo^\c(y)$ are essentially the modeling error at the interface and in the continuum region. If we assume that the computational domain contains only a single atomistic region $\Omega^{\a}$ whose radius is denoted by $R_{\a}$, then the following theorem states that for $R_{\rm a}$ sufficiently large, $\zeta_\mo^\c(y)$ is negligible compared with $\zeta_\cg(y)$ and their elementwise counterparts obey a certain ratio relation.



%

\begin{theorem}\label{thm:main}
Suppose Assumption \ref{as:local} holds and the computational domain only contains a single atomistic region. Let $y$ be a strongly stable solution of \eqref{eq:min}, $u = y - y^{B}$, and $\zeta_\mo^\c(y)$ and $\zeta_\cg(y)$ be defined in \eqref{eq:zeta}. If the triangulation $\T_{h}$ is a graded mesh such that the mesh size function $h(x)$ for $x\in T\in \T_{h}$ satisfies
\begin{equation}\label{eq:h}
	|h(x)|\leq C^{\rm{mesh}}  \b(\frac{|x|}{R_\a}\b)^\beta, \quad  \text{  with  }1<\beta<\frac{d+2}{2},
\end{equation}
we then have
\begin{align}\label{eq:lima}
\lim_{R^{\a} \rightarrow \infty} \frac{\zeta^\c_{\mo}(y)}{\zeta_\cg(y)} = 0.
\end{align}
Moreover, for $T\in \T_{h}$, if we define the local error contributions by $\zeta_\cg(y; T):=\|h D^2u\|_{\ell^2(\Lambda_T)}$ and $\zeta^\c_{\mo}(y; T):=\|D^3u\|_{\ell^{2}(\Lambda_T)} + \|D^2u\|^2_{\ell^{4}(\Lambda_T)}$ respectively with $\Lambda_T:=\Lambda \cap T$. Then for $R_{\a}$ sufficiently large, there exits a constant $C^{a-priori}_{\mo-\cg}>0$ such that
\begin{equation}\label{eq:local}
    \frac{\zeta^\c_{\mo}(y; T)}{\zeta_\cg(y; T)} \leq \frac{C^{a-priori}_{\mo-\cg}}{h_{T}},
\end{equation}
where $h_T := \mathrm{diam}(T)$ for $T\in \T_{h}$.
\end{theorem}

\begin{proof}
Let $y$ be a strongly stable solution of \eqref{eq:min} and $u = y - y^{B}$. We have the following sharp decay estimates of $u$ for both point defects and dislocations \cite[Theorem 2.3, 3.6]{2013-defects}, for $|\ell|$ sufficiently large and $j=1,2,3$, 
\begin{align}\label{eq:ubar-decay}
\big|D^j u(\ell)\big| &\leq C \big(1+|\ell|\big)^{-d-j+1}, \qquad \quad  \textrm{point defects}, \nonumber \\
\big|D^j u(\ell)\big| &\leq C\big(1+|\ell|\big)^{-j-1}\log(|\ell|), \qquad \textrm{dislocations},
\end{align}
where the stencil norm is defined by $|D u(\ell)| := \big(\sum_{\rho\in\Rg_{\ell}}D_{\rho}u(\ell)\big)^{1/2}$. For ${\bm \rho}\in (\Rg_{\ell})^j$, $D_{\bm \rho}u:=D_{\rho_1}\cdots D_{\rho_j}u$ denotes a $j$-th order derivative, and $D^ju$ is defined recursively by $D^ju:=DD^{j-1}u$ which denotes the $j$-th order collection of derivatives.

By the definition of $\zeta_\mo^{\c}(y)$, $\zeta_\cg(y)$, together with the sharp decay estimates \eqref{eq:ubar-decay} and the restriction of the mesh size function $h(x)$ by \eqref{eq:h}, for $R_{\rm a}$ sufficiently large, which is equivalent to $|\ell|$ sufficiently large, we obtain for point defects,
\begin{align}
    \zeta_\mo^{\c}(y) &\leq C_{\zeta_\mo} \Big( \int_{R_{\a}}^{\infty} r^{d-1} (r^{-d-2})^2 \dr \Big)^{1/2} \leq C_{\zeta_\mo} (R_{\a})^{-d/2-2}, \nonumber \\
    \zeta_\cg(y) &\leq C_{\zeta_\cg} \Big( \int_{R_{\a}}^{\infty} r^{d-1} \Big(\frac{r}{R_{\a}}\Big)^{2\beta} (r^{-d-1})^2 \dr \Big)^{1/2} \leq C_{\zeta_\cg} (R_{\a})^{-d/2-1}. 
\end{align}
For dislocations, similarly we have the estimates \cite[Proposition 2]{2013-defects} 
\begin{align}
    \zeta_\mo^{\c}(y) \leq C'_{\zeta_\mo}  (R_{\a})^{-2} \quad \text{and} \quad
    \zeta_\cg(y) \leq C'_{\zeta_\cg} (R_{\a})^{-1}.
\end{align}
Hence, we have 
\[
\lim_{R^{\a} \rightarrow \infty} \frac{\zeta^{\c}_{\mo}(y)}{\zeta_\cg(y)} = \lim_{R^{\a} \rightarrow \infty} C^{a-priori}_{\mo-\cg} R_{\a}^{-1} = 0,
\]
with $C^{a-priori}_{\mo-\cg}:=\max\{C_{\zeta_\mo} /C_{\zeta_\cg}, C'_{\zeta_\mo} /C'_{\zeta_\cg}\}$, which yields the stated result \eqref{eq:lima} for both point defects and dislocations.

Let $\T_{\a}$ be the {\it canonical} triangulation induced by $\L$. Given $T\in \T^\c_{h}$, we define the set of ``inner finest elements" by $
\T_{\varepsilon}^{T}:= \{T_{\varepsilon} \in \T_{\a}~\big|~ |T_{\varepsilon} \cap T| = |T_{\varepsilon}|\}
$. 

To obtain the estimate \eqref{eq:local}, we first have for any $T\in \T^\c_{h}$ and $T_{\varepsilon} \in \T^T_{\varepsilon}$
\begin{equation}\label{eq:localest}
    \frac{\zeta^{\c}_{\mo}(y; T)}{\zeta_\cg(y; T)} \leq C \frac{\|D^3 u\|_{\ell^2(\Lambda_T)}}{\|h D^2u\|_{\ell^2(\Lambda_{T})}} \leq C^{a-priori}_{\mo-\cg} h^{-1}_T \frac{\|D^3 u\|_{\ell^2(\Lambda_{T_\varepsilon})}}{\| D^2u\|_{\ell^2(\Lambda_{T_\varepsilon})}}.
\end{equation}
Denoting the barycenter of $T_{\varepsilon}$ by $x_{T_\varepsilon}$ and applying the decay estimates \eqref{eq:ubar-decay} again, we arrive at
\begin{eqnarray}\label{eq:localepsilon}
    \frac{\zeta^{\c}_{\mo}(y; T)}{\zeta_\cg(y; T)} \leq C^{a-priori}_{\mo-\cg}h^{-1}_{T} \frac{(x_{T_\varepsilon}+1)^{-3} - x_{T_\varepsilon}^{-3}}{(x_{T_\varepsilon}+1)^{-2} - x_{T_\varepsilon}^{-2}} \leq \frac{C^{a-priori}_{\mo-\cg}}{h_T},
\end{eqnarray}
where in the last inequality we use the fact $x_{T_\varepsilon}>1$. Combining \eqref{eq:localest} and \eqref{eq:localepsilon} we achieve the estimate~\eqref{eq:local} which finishes the proof of Theorem \ref{thm:main}.
\end{proof}


We then give a similar {\it a posteriori} estimate under suitable assumptions which is computable and can be employed in the adaptive computations. We note that similar to $\zeta^\c_{\mo}(y; T)$, $\eta^\c_\mo(y_h; T)$ represents the elementwise residual in the continuum region.

\begin{corollary}\label{coro:ap}
Under the assumptions of Theorem \ref{thm:main}, if $y_h$ is a strongly stable solution of \eqref{eq:min_ac} satisfying \eqref{eq:ap} and we additionally assume that $u_h$ satisfies the decay estimates \eqref{eq:ubar-decay}, then for $T\in \T_{h}$, there exists constants $C^{\rm a-post}$ and $C^{\rm a-post}_{\rm mo-cg}$ such that
\begin{equation}\label{claim}
    \frac{\eta^\c_\mo(y_h; T)}{\eta_\cg(y_h; T)} \leq C^{\rm a-post} \frac{\zeta^\c_\mo(y; T)}{\zeta_\cg(y; T)} \leq \frac{C^{\rm a-post}_{\rm mo-cg}}{h_T}.
\end{equation}
\end{corollary}


\begin{proof}
We sketch out the proof here for the sake of simplicity. If $u$ and $u_h$ are the strongly stable solutions of \eqref{eq:min} and \eqref{eq:min_ac} respectively satisfying \eqref{eq:ap}, for $R_\a$  sufficiently large and $T$ lying in the region that far away from the defect core, we can roughly (abuse of notations) estimate that $\sigma^\a(u;T)/\sigma^\ac(u;T) \lesssim  \sigma^\a(u_h;T)/\sigma^\ac(u_h;T)$. Then, following the {\it a priori} consistency estimates presented in \cite[Section 5]{PRE-ac.2dcorners} and the {\it a posterior} residual estimates given in Theorem \ref{thm:h1error}, we have an asymptotic result that $\eta^{\c}_{\mo}(y_h; T)/\eta_\cg(y_h; T) \leq C^{\rm a-post} \zeta^{\c}_{\mo}(y; T)/\zeta_\cg(y; T)$ under an additional assumption of $u_h$ satisfying the same decay estimates as $u$. According to the estimate of $ \zeta^{\c}_{\mo}(y; T)/\zeta_\cg(y; T)$ in Theorem~\ref{thm:main}, we obtain the stated result by choosing $C^{\rm a-post}_{\rm mo-cg}=C^{a-priori}_{\mo-\cg}C^{\rm a-post}$, where $C^{a-priori}_{\mo-\cg}$ is introduced in \eqref{eq:local}.
\end{proof}

We note that there are three key assumptions for the analytical result \eqref{claim} in Corollary \ref{coro:ap} to hold. The first assumption is that $u_h$ shares the same decay estimates as $u$ which results in the first inequality of \eqref{claim}. Such assumption is not easy to prove rigorously in principle and we refer the interested reader to the analysis in \cite[\S~5.1]{wang2020posteriori} in the context of QM/MM coupling methods. The second and the third assumptions are that only a single atomistic region $\Omega^{\a}$ exists and $R_{\a}$ is sufficiently large which result in the {\it a priori} decay estimate \eqref{eq:ubar-decay} and consequently the second inequality of \eqref{claim} (or essentially \eqref{eq:local}). The second assumption is certainly violated by the defect of multiple vacancies in the first steps of the adaptive computation where multiple disjoint atomistic regions exist (see the first four subplots in Figure \ref{figs:evolution} in \S~\ref{sec:numerics:multiple vacancies}). The third assumption may not be satisfied by all three cases of defects in the first steps of the adaptive computations and the quantitative characterization of how large $R_\a$ should be in order to guarantee the decay estimates may involve substantial additional technicalities. 

Therefore, instead of pursuing the theoretical consideration further, we verify the first assumption and the estimate \eqref{claim} numerically in Figure \ref{fig:decay_rates_uh} and Figure \ref{figs:ratio_of_error_estimator} respectively whose further detail can be found later in \S~\ref{sec:numerics}.

\begin{figure}[htb]
	\centering 
	\subfloat[Micro-crack]{
		\label{fig:decay_mcrack}
		\includegraphics[height=4.0cm]{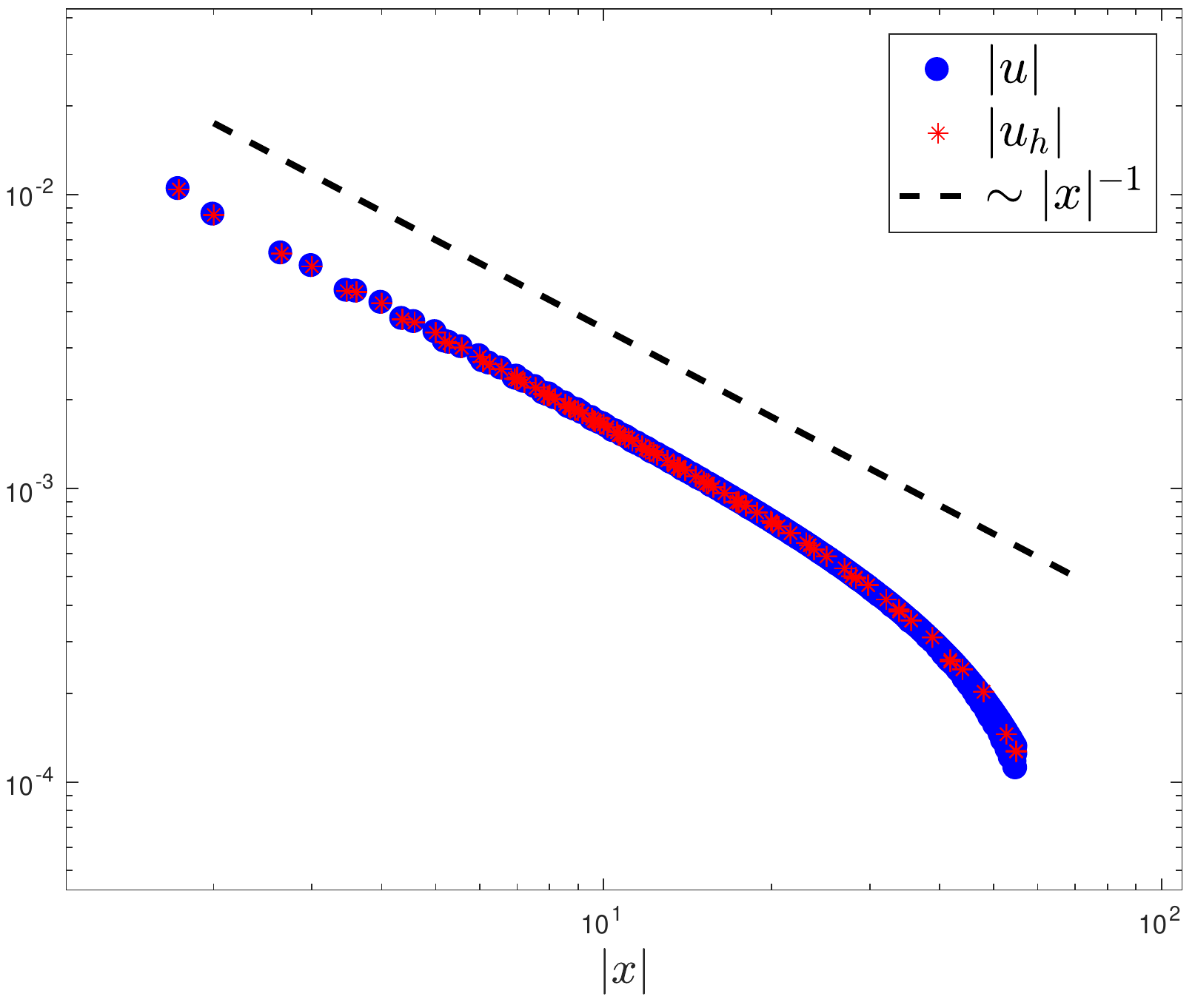}}
	\hspace{0.2cm} 
	\subfloat[Anti-plane screw dislocation]{
		\label{fig:decay_screw}
		\includegraphics[height=4.0cm]{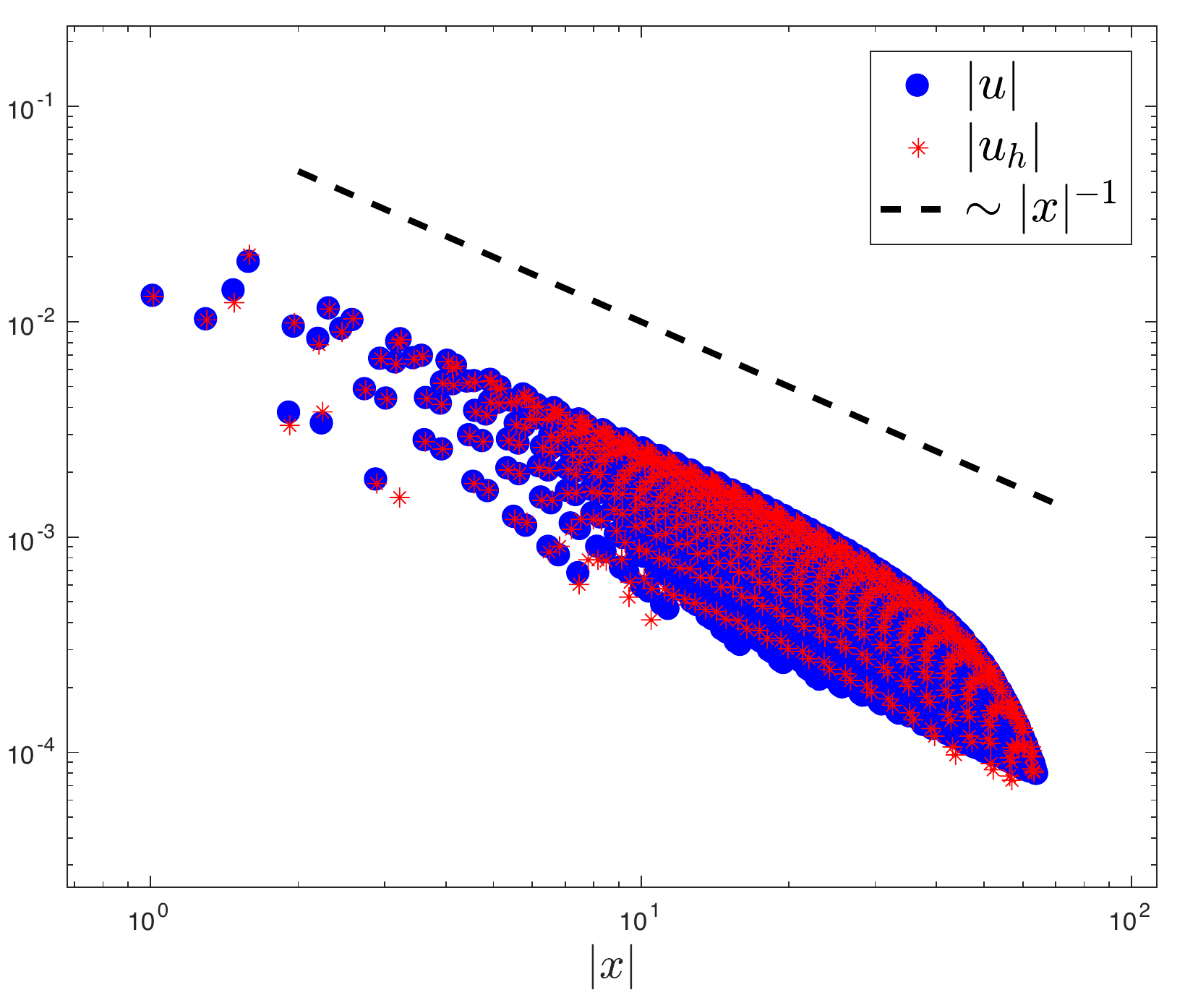}}
	\hspace{0.2cm} 
		\subfloat[Multiple vacancies]{
		\label{fig:decay_multivac} 
		\includegraphics[height=4.0cm]{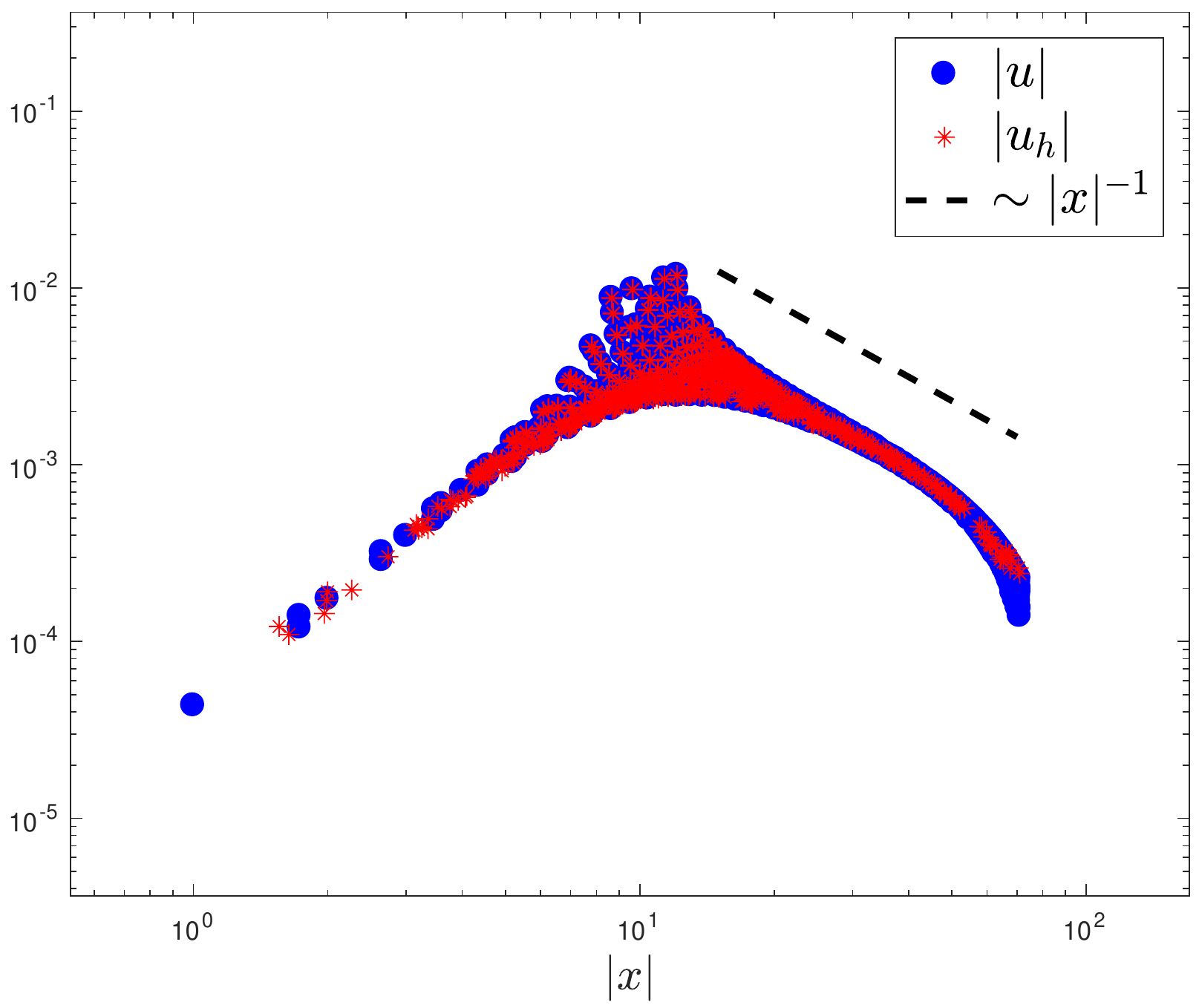}}
	\caption{Numerical verification of the decay rate of $u$ and $u_h$.}
	\label{fig:decay_rates_uh}
\end{figure}

The blue dots in each subplot in Figure \ref{fig:decay_rates_uh} are the modules of the atomistic solution $u$ sampled in six different directions in a spokewise form from the defect core $\Omega^\dfct$ while the red asterisks are taken from the a/c solution $u_h$ in one of the adaptive steps when $R_{\a}\ge 10$ for the first time for micro-crack and anti-plane screw dislocation and when $R_{\a}\ge 8$ (or more precisely, the radius of each atomistic region is greater than or equal to $8$) for the first time for multiple vacancies while the atomistic regions are still disjoint. The $x$-axis denotes the distance from the ``center" of the defects. We note that we define the ``center" of the defect core of the separated vacancies as their geometric center and an increase of $|u|$ and $|u_h|$ appears in the first half of the $x-$axis in the figure since it gets closer to the true defects as the distance $|x|$ grows initially. We observe in all three cases that $u_h$ essentially follows the same decay (or increase) rates as $u$ just as we assumed in Corollary \ref{coro:ap}. In addition, we observe a close match between $u$ and $u_h$ now but the detailed convergence of the error together with a thorough discussion is presented in \S~\ref{sec:numerics}. 


\begin{figure}[htb]
\centering
	\subfloat[Micro-crack]{
		\label{fig:ratio_mcrack}
		\includegraphics[height=4.0cm]{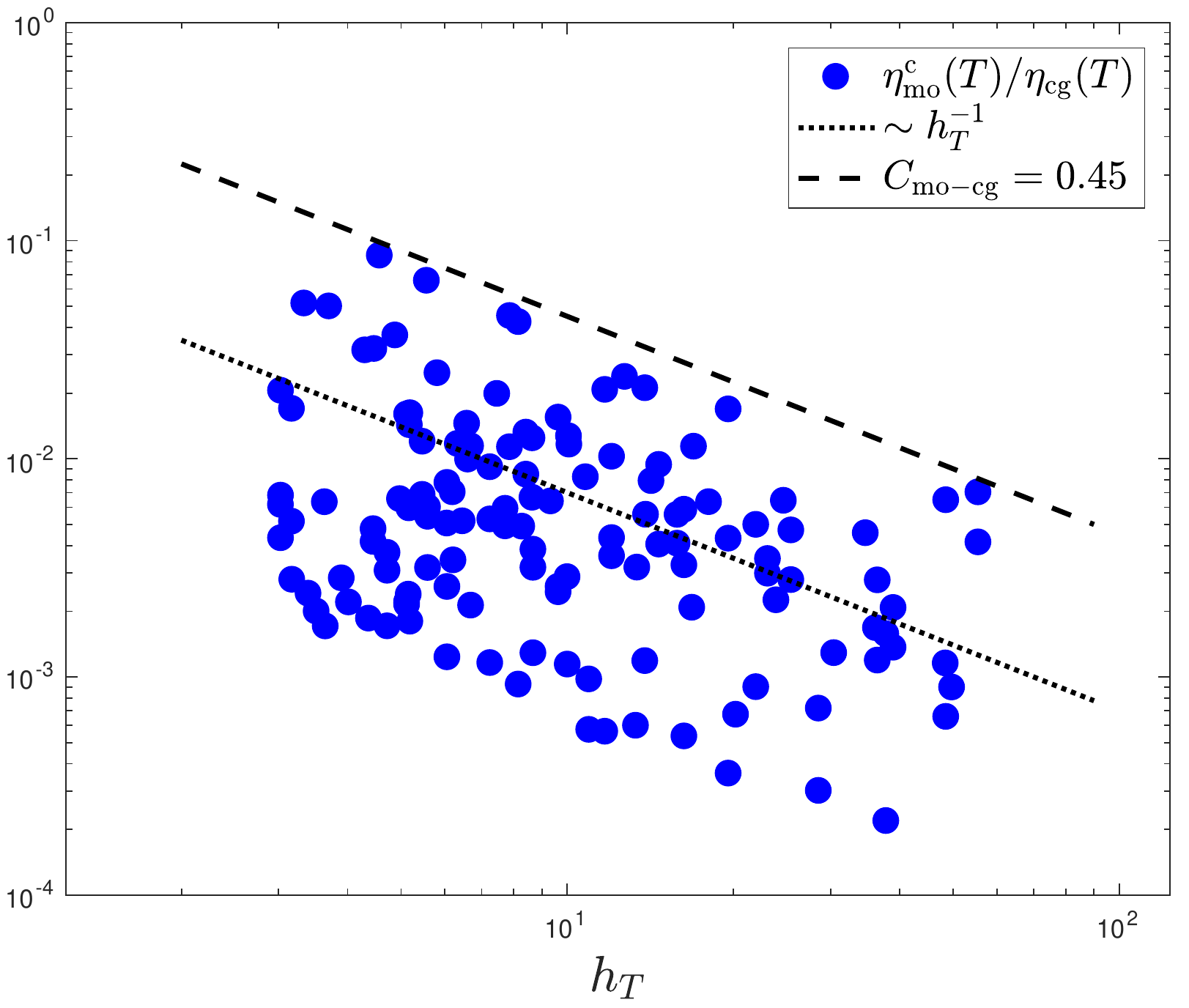}}
	\hspace{0.2cm} 
	\subfloat[Anti-plane screw dislocation]{
		\label{fig:ratio_screw}
		\includegraphics[height=4.0cm]{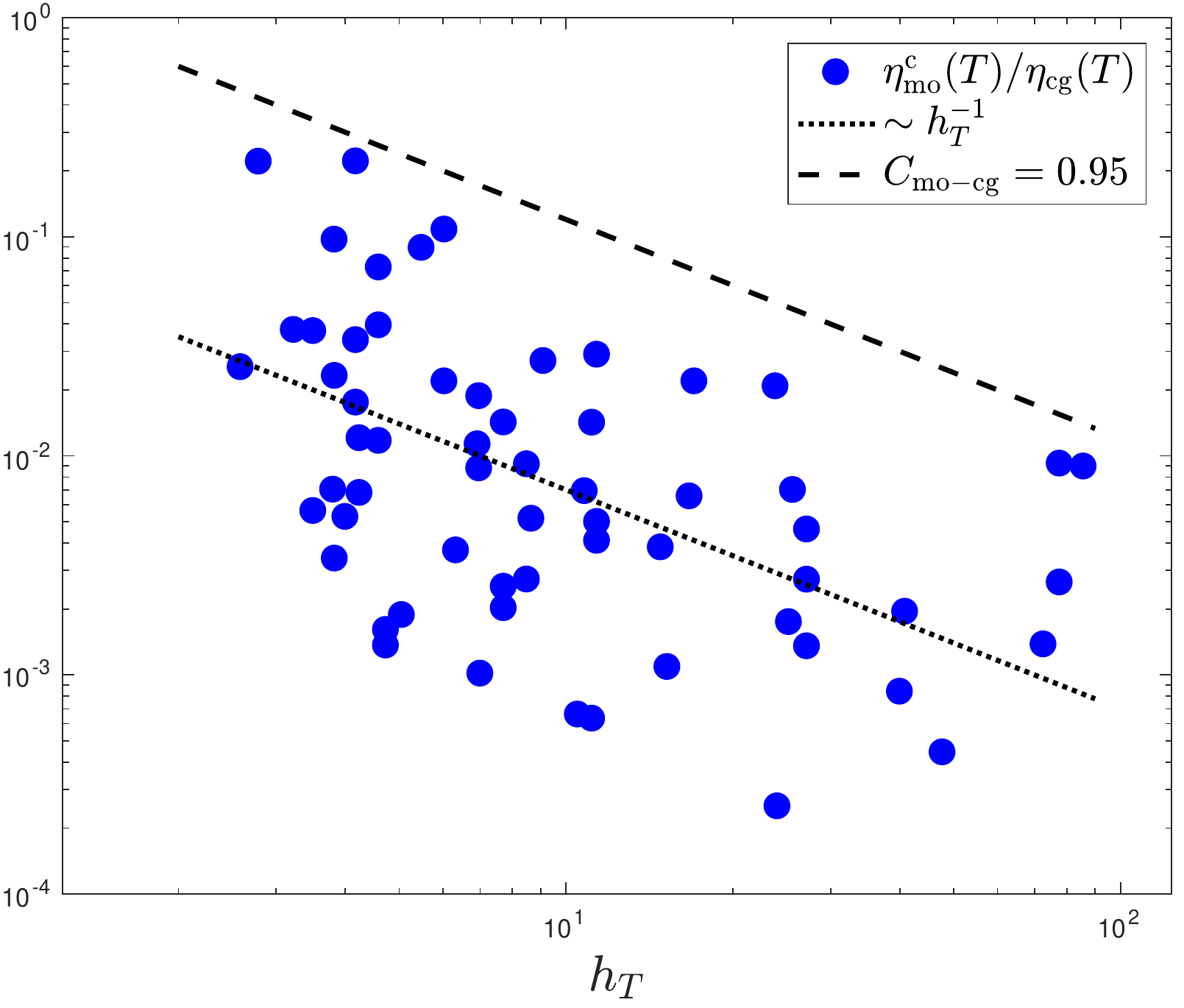}}
	\hspace{0.2cm} 
		\subfloat[Multiple vacancies]{
		\label{fig:ratio_multivac} 
		\includegraphics[height=4.0cm]{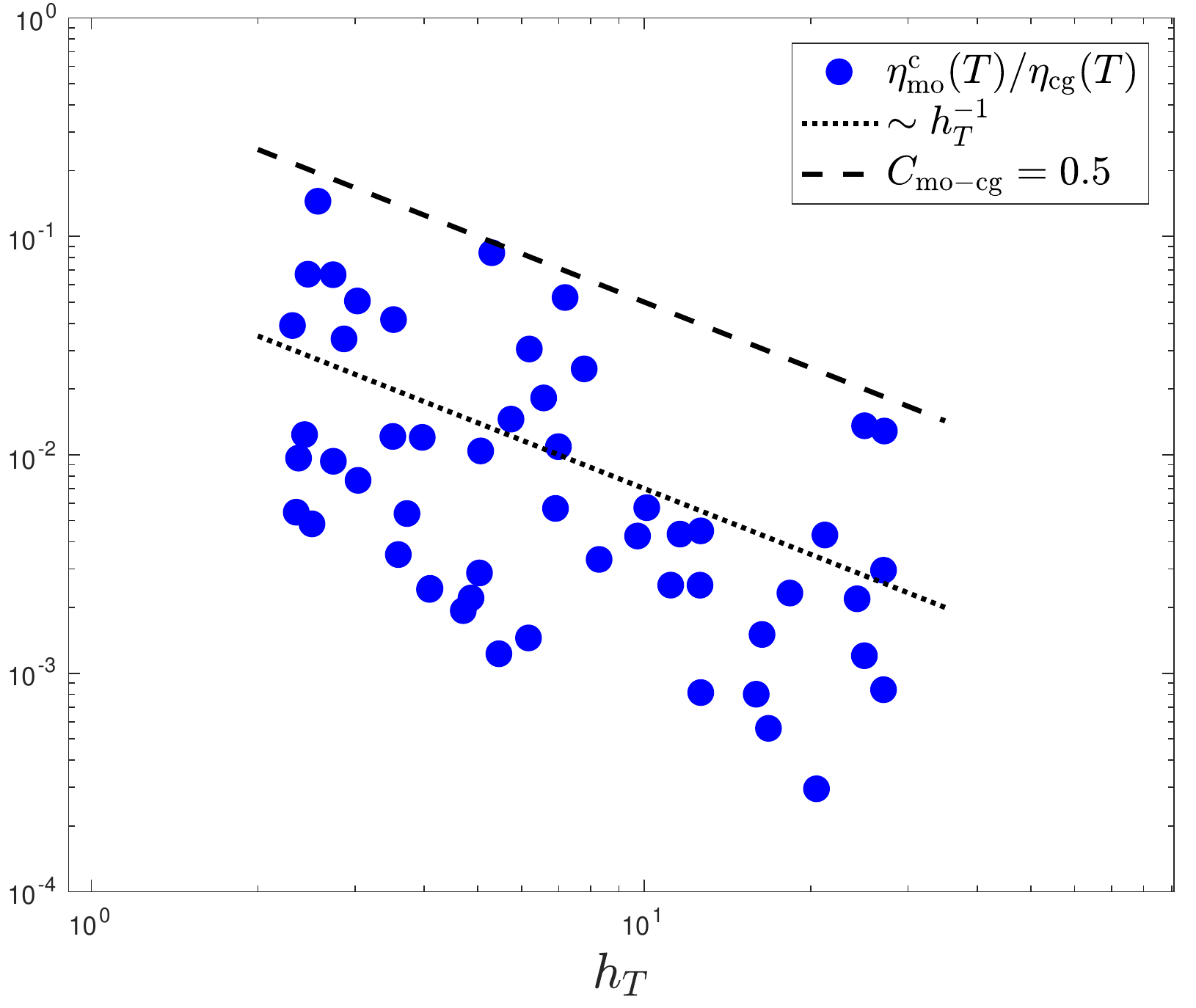}}
	\caption{Elementwise modeling error over elementwise coarsening error versus the diameter of corresponding element.}
	\label{figs:ratio_of_error_estimator}
\end{figure}

Figure \ref{figs:ratio_of_error_estimator} directly verifies the estimate in \eqref{claim} in Corollary \ref{coro:ap}. The three subplots are the snapshots at the same mesh as those in Figure \ref{fig:decay_rates_uh}. We clearly see that with the increasing diameters of the elements in the continuum region, the ratio of the elementwise modeling error over the elementwise coarsening error, namely $\eta^{\c}_{\mo}(y_h; T)/\eta_\cg(y_h; T)$, decays asymptotically with the rate $h^{-1}_T$. Though the exact value of $C^{\rm a-post}_{\rm mo-cg}$ is difficult to obtain, we use the largest {\rm mo-cg} ratio for the elements $T$ near $\Omega^\a \cup \Omega^{\rm buf}$ as the estimate of $C^{\rm a-post}_{\rm mo-cg}$ in practice, where $\Omega^{\rm buf}$ is the buffer region whose definition is given in \eqref{def: buffer_region}. The upper straight line validates our choice of this constant as all the points are under this line. 




\section{Modified Local Error Estimators and the Adaptive Algorithms}
\label{sec:eff_adapt_algo}

In this section, we propose several {\it modified} residual based local error estimators based on the theoretical results derived in the last section and design the algorithm for the {\it a posteriori} error control and adaptive computations for the system we consider. We note that our adaptive procedure is the classic one, namely
$$
\mbox{Solve}~\rightarrow~\mbox{Estimate}~\rightarrow~
\mbox{Mark}~\rightarrow~\mbox{Refine},
$$
and a common practice is to include the adaptive algorithm and the numerical results in one section. However, since the analysis in \S~\ref{sec:local} gives us plenty of possibilities to design various {\it modified} error estimators and mesh structures, we decide to separate out the elaboration of these estimators and mesh structures in this independent section so that the ideas behind are well understood.

\subsection{Stress tensor correction}
\label{sec:Interfacial-stress}
Before we go into the detail of the modified estimators and mesh structures, we first briefly review the stress tensor correction which is a preliminary step for specifying a correct {\it a posteriori} error estimator.


As we mentioned in \S~\ref{sec:formulation:stress}, the stress tensors $\sh$ and $\sa$, on which $\eta_\trc$, $\eta_\mo$, and $\eta_\cg$ all depend, are unique only up to a {\it divergence free} tensor field whose definition is given in \eqref{eq: divergence_free_condition}. Therefore, we need to minimize $\eta(y_h) = \eta_{\rm tr}(y_h)+\eta_{\rm mo}(y_h)+\eta_{\rm cg}(y_h)$ with respect to all the admissible stress tensors such that 
\begin{equation}
		 \< \del\Ea(\Ia y_h), v \> \leq  \min_{c_{\rm  a}\in N_1(\T_{\rm a})^2, c_h\in N_1(\T_h)^2}\tilde{\eta}\big(\sa(\Ia y_h)+\nabla c_{\rm a}\mJ, \sh(y_h)+\nabla c_h\mJ \big)\|\nabla v\|_{L^2},
		\label{eq:sharpresidual2}
\end{equation}
where $\mJ$ is the rotation matrix defined in \cite[\S~2.4.3]{wang2018posteriori}. Instead of minimizing the total error estimator $\eta$ with respect to $c_{\rm a}$ and $c_h$ introduced in \eqref{eq:sharpresidual2}, \cite{wang2018posteriori} introduced an approximation of the stress tensor correction which only requires minimizing the modeling error $\eta_\mo$ with respect to the degrees of freedom of $\sh$ adjacent to the interface. This method dramatically reduces the computational cost of such correction but greatly improves the accuracy of the estimator in the adaptive computation.

For the sake of completeness, we give the stress tensor correction algorithm which essentially follows from \cite[Algorithm 1]{wang2018posteriori}. 

\begin{algorithm}
\caption{Stress tensor correction}
\label{alg:etc}
\begin{enumerate}
\item Take $\sa(\Ia y_h)$ and $\sh(y_h)$ as the canonical forms in \eqref{eq:defn_Sa} and \eqref{eq:defn_Sh} respectively. 
\item Denote $q_f$ as the midpoint of $f\in \Fh$. $c_h$ minimizes the following sum
\begin{equation}
	\sum_{T\in\T^\i}|T|\left[\sa(\Ia y_h, T) - \big(\sh(\Ia y_h, T)+\nabla c_h \mJ\big)\right]^2
\end{equation} 
subject to the constrain such that $c_h(q_f)=0$, for $f\bigcap \L_\i = \emptyset$.
\item Let $\sh(y_h)  = \sh(y_h) + \nabla c_h\mJ$, compute $\eta_{\rm mo}$, $\eta_{\rm tr}$ and $\eta_{\rm cg}$ with $\sa(\Ia y_h)$ and $\sh(y_h)$.
\end{enumerate}
\end{algorithm}

\subsection{Modified local error estimators and the adaptive algorithm}
\label{sec: local_estimator_and_adaptiv_algorithm}
We are finally at the stage of proposing various modified local {\it a posteriori} error estimators and design our main adaptive algorithm. Before that, we should first recall the difficulty we face and the theoretical results in hand that may be able to deal with the difficulty so that the ideas behind the modified error estimators are easily understood.

Our main difficulty comes from the overwhelming cost of computing the elementwise modeling residual defined in \eqref{eq:localetam} which is due to the mismatch of $\T_{\a}$ and $\T_h$ in $\Omega^\c$ as illustrated in Figure \ref{figs:TcapTh}. We then prove in Theorem \ref{thm:main} and Corollary \ref{coro:ap} that, under suitable assumptions, the elementwise modeling residual is proportional to the elementwise coarsening residual by a factor of $O({h^{-1}_T})$ in the continuum region, where $h_T$ is the diameter of the element. Moreover, we numerically verify that this ratio relationship holds for the cases of defects we consider even without the stringent assumptions in the theorem and corollary. We may conclude that both the theoretical analysis and the numerical evidence show that, in principle, we are able to approximate the elementwise modeling residual by its coarsening counterpart. However, we should be conservative to such approximation around the atomistic-to-continuum interface, since a good number of researches in the analysis of the a/c coupling method have shown that the interface error may be large and subtly important for the construction of an efficient a/c mesh \cite{LuOr:acta, PRE-ac.2dcorners, wang2018posteriori}.


To justify our concern, we show in Figure \ref{figs:different_parts_error_estimator} the decay of the modeling residual and the coarsening residual given in \eqref{eqn:etamo} and \eqref{eqn:etacg} respectively and those in different regions, with respect to the increasing number of degrees of freedom. The computations use the {\it original} elementwise modeling residual and its coarsening counterpart defined in \eqref{eq:localetam} and \eqref{eq:localetaC} respectively. Figure \ref{figs:different_parts_error_estimator} shows that, somehow as we predicted, the overall modeling residual (represented by the red line with asterisk markers) is almost negligible (about two orders smaller) compared with the overall coarsening residual (represented by the dark blue line with round markers). However, the modeling residual (represented by the purple line with square markers) and its coarsening counterpart (represented by the light blue line with triangle markers) are comparable around the atomistic-to-continuum interface which affirms that the modeling residual may not be neglected or even be approximated at the interface. In addition, we clearly observe that the modeling residual concentrates at the interface which perfectly coincides with various earlier analysis of the interface influence of different coupling schemes \cite{LuOr:acta, LiOrShVK:2014, OrZh:2016}.  The numerical observations listed above essentially motivate us to propose the following constructions of the modified {\it a posteriori} error estimators and mesh structures. 



\begin{figure}[htb]
\centering
	\subfloat[Micro-crack]{
		\label{fig:diff_ratio_mcrack}
		\includegraphics[height=3.8cm]{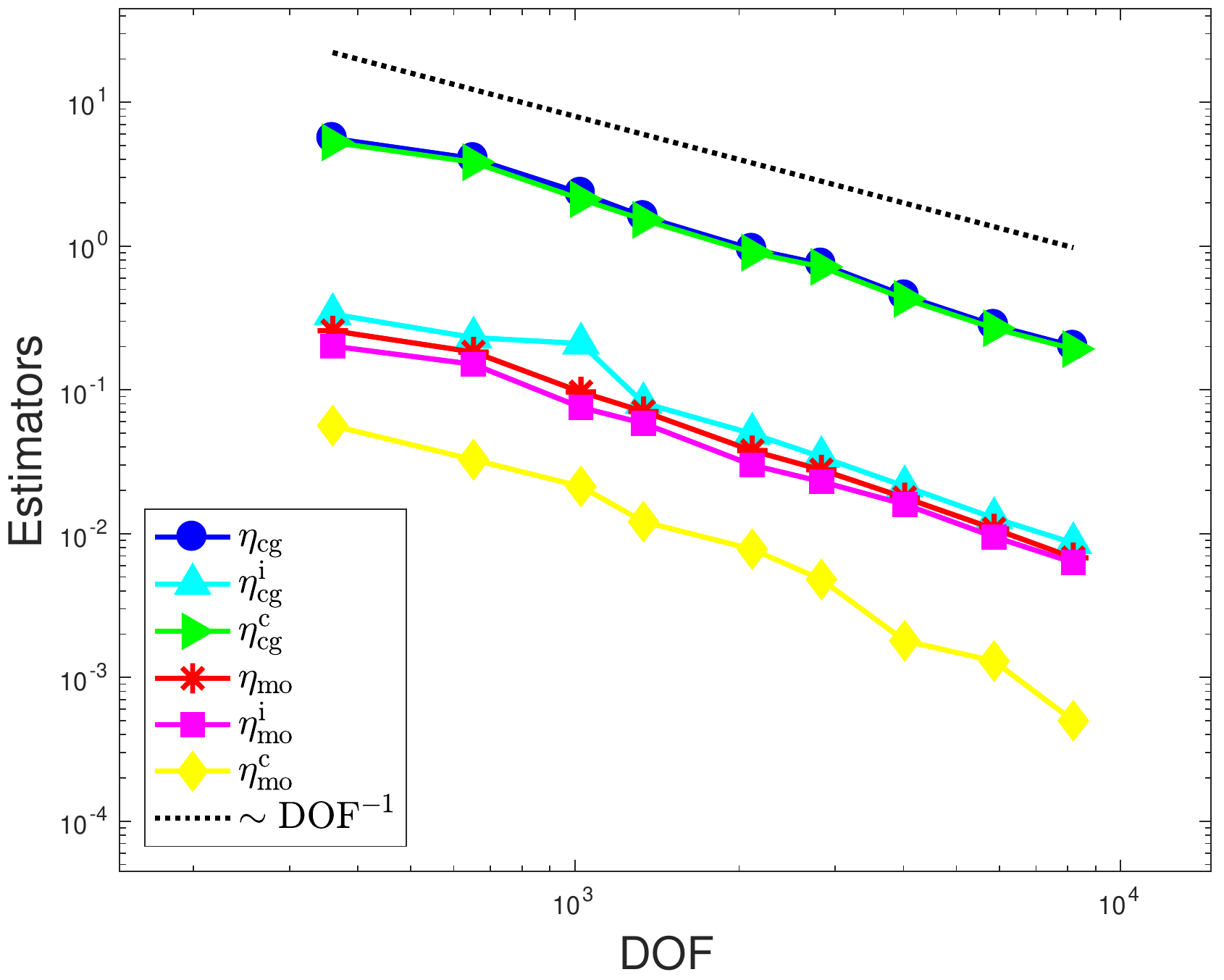}}
	\hspace{0.2cm} 
	\subfloat[Anti-plane screw dislocation]{
		\label{fig:diff_ratio_screw}
		\includegraphics[height=3.8cm]{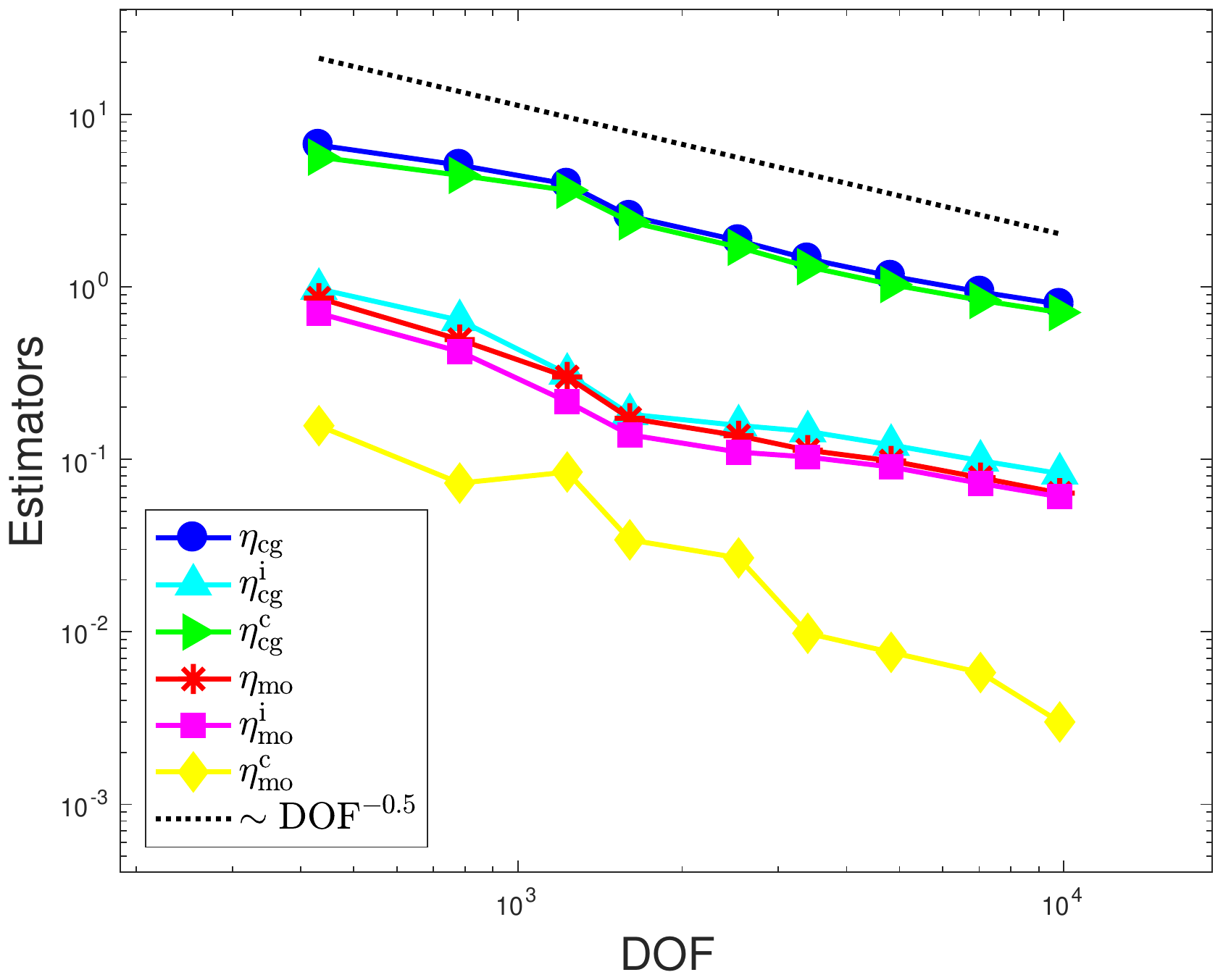}}
	\hspace{0.2cm} 
		\subfloat[Multiple vacancies]{
		\label{fig:diff_ratio_multivac} 
		\includegraphics[height=3.8cm]{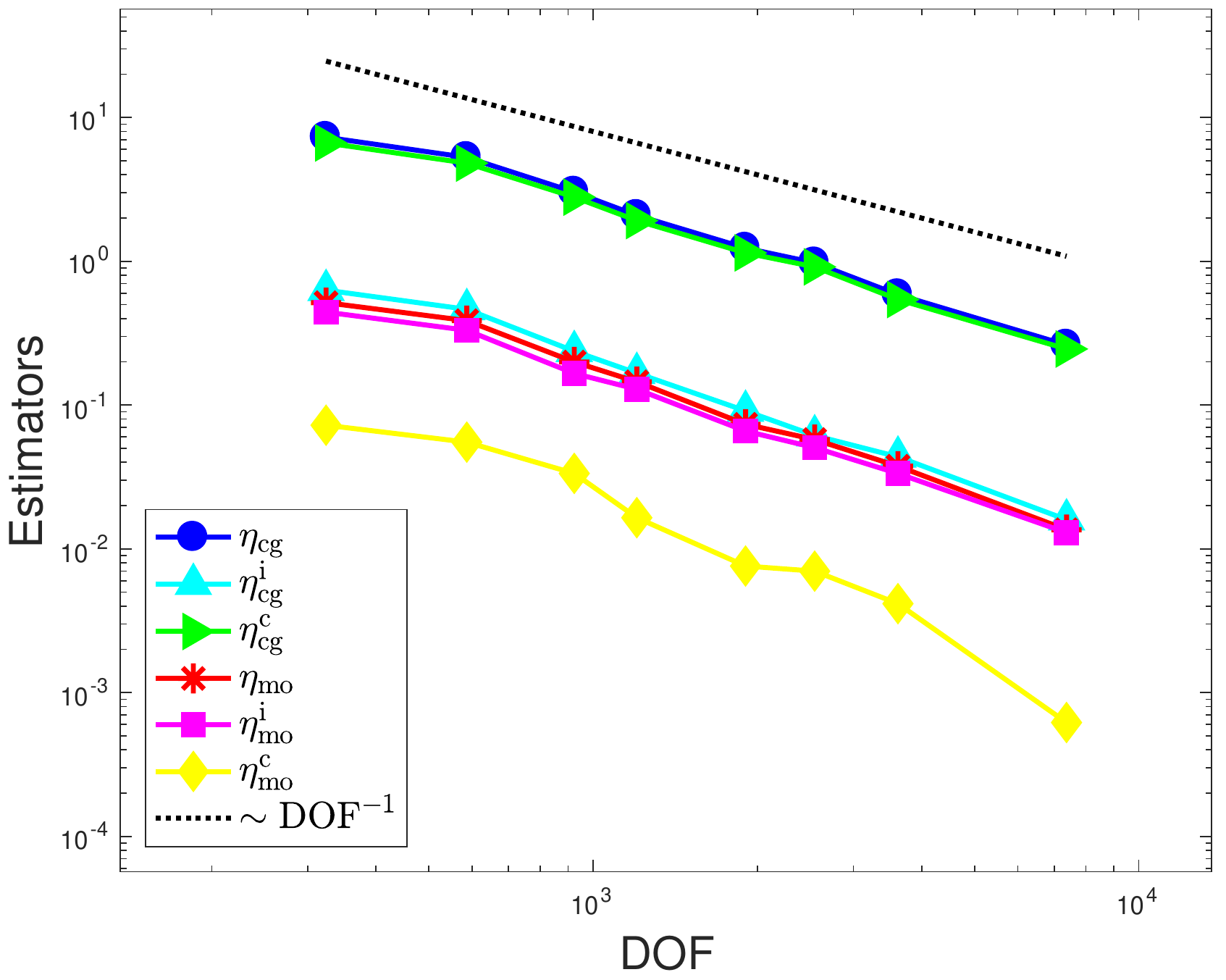}}
	\caption{The modeling residual and the coarsening residual in different regions.}
	\label{figs:different_parts_error_estimator}
\end{figure}

To keep the accuracy of the interface modeling residual while reducing the computational cost, we introduce an ``interface buffer" region $\Omega^{\rm buf}$ such that 
\begin{eqnarray}
\label{def: buffer_region}
\Omega^{\rm buf} := \bigcup_{T \in \T^{\rm buf}_{h}} T \qquad \textrm{with} \quad \T^{\rm buf}_{h}:=\{T \in \T^{\c}_{h} ~|~ \exists T' \in \T_{\a}~~\textrm{s.t.}~~T \cap T' = T'\}.
\end{eqnarray}
We denote $\L^{\rm buf}:= \L \cap \Omega^{\rm buf}$. See Figure \ref{figs:plotMesh} for an illustration of the structure of the mesh $\T_h$ and the ``interface buffer" region $\Omega^{\rm buf}$ (the grey region) for a single vacancy. 
\begin{figure}[htb]
\begin{center}
	\includegraphics[scale=0.45]{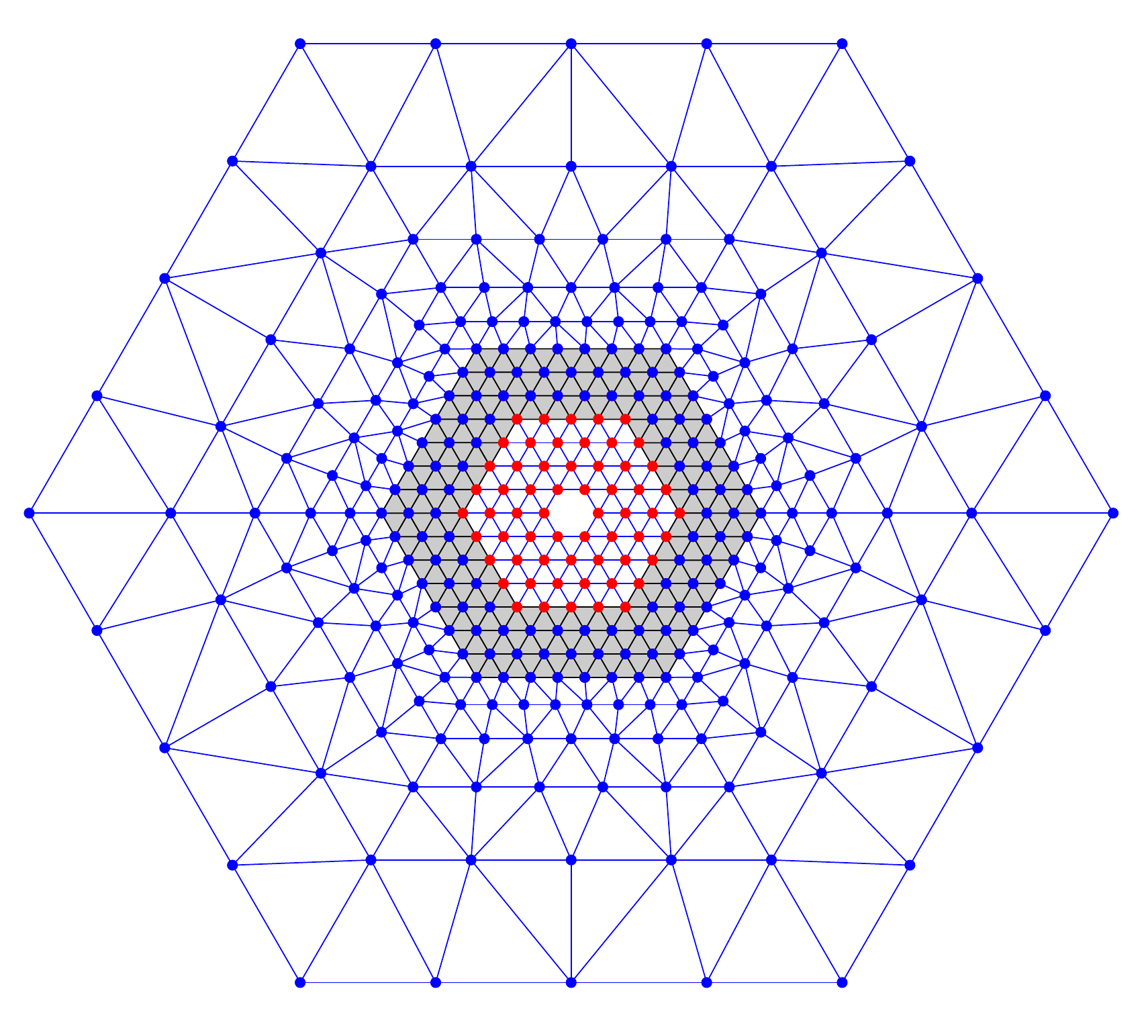}
	\caption{Illustration of ``interface buffer" region $\Omega^{\rm buf}$ which is the domain marked grey.}
	\label{figs:plotMesh}
\end{center}
\end{figure}

We note that the modeling residual can be computed exactly by \eqref{eq:localetam} in $\Omega^{\rm buf}$ without the need of finding the geometric relation of $\T_{\a}$ and $\T_h$ as they coincide in this region and thus the save of the computational cost. However, the length of $\Omega^{\rm buf}$, $R_{\rm buf}$ should be kept small since it will introduce unnecessary degrees of freedom otherwise.

%

\def\bld{{\rm bld}}

Outside $\Omega^{\rm buf}$, namely, the region $\Omega \setminus (\Omega^{\rm a} \cup \Omega^{\rm buf})$, we apply certain modification for the elementwise {\it a posteriori} error estimator. A smooth way to introduce the modification is to employ a blending function $\beta\in C^0(\R;\R)$ so that the {\it blended approximated} modeling residual $\tilde{\eta}_\mo(u_h; T)$ for $T \in \T_h^{\rm c}\setminus \T_h^{\rm buf}$ is defined by 
\begin{eqnarray}\label{eq:approx}
    \tilde{\eta}_\mo(u_h; T) := \beta(r^{\rm buf}_T)\cdot\eta_\mo(u_h; T) + \big(1-\beta(r^{\rm buf}_T)\big)\cdot \frac{C^{\rm a-post}_{\rm mo-cg}}{h_{T}}\eta_\cg(u_h; T),
\end{eqnarray}
where $r^{\rm buf}_{T} := {\rm dist}(T, \L^{\rm buf}) := \inf \{|\ell-x_T|, \forall \ell \in \L^{\rm buf}\}$ is the distance between $T$ and $\L^{\rm buf}$ and $x_T$ is the barycenter of $T$. $C^{\rm a-post}_{\rm mo-cg}$ is the constant defined in \eqref{claim} which is in practice determined empirically (cf. Figure \ref{figs:ratio_of_error_estimator}). Given the blending width $R_{\rm bld}>0$, a typical linear blending function $\beta\in C^0(\R;\R)$ is defined by
\begin{eqnarray}
\label{beta}
	\beta(r) :=
	\left\{ \begin{array}{ll}
		1, & r \geq R_{\rm bld}
		\\
		\frac{r-1}{R_{\rm bld}-1}, & 1 < r < R_{\rm bld}
		\\
		0, & r \leq 1 
	\end{array} \right. .
\end{eqnarray}

A more straightforward way for approximating the elementwise modeling error residual for $T \in \T^{\rm c}_h \setminus \T_h^{\rm buf}$ is to directly define 
\begin{eqnarray}\label{eq:direct_approx}
    \tilde{\eta}_\mo(u_h; T) :=  \frac{C^{\rm a-post}_{\rm mo-cg}}{h_{T}}\eta_\cg(u_h; T).
\end{eqnarray}
The advantage of such approximation is apparent that no geometric consideration is involved even for a small layer of elements near the ``interface buffer" region.

\def\M{\mathcal{M}}



We denote $\tilde{\eta}^2_{\rm mo}(u_h):=(C^{\rm tr})^2 \sum_{T\in\T_h} \tilde{\eta}^2_{\rm mo}(u_h;T)$. Recall that the local coarsening residual is already defined by \eqref{eq:localetaC}. We thus are ready to define the {\it modified} elementwise local estimator $\rho_{T}(u_h;T)$ for $T \in \T_h$ as
\begin{equation}
\rho_{T}(u_h;T) := (C^{\rm tr})^2 \frac{\tilde{\eta}_{\rm mo}^2(u_h; T)}{\tilde{\eta}_{\mo}(u_h)}+(\sqrt{3}C^{\rm tr}C_{\T_h})^2\frac{\eta^2_{\cg}(u_h;T)}{\eta_\cg(u_h)}.
\label{eq:localestimator1}
\end{equation} 
It is straightforward to see that $\sum_{T \in \T_h} \rho_T(u_h;T) = \tilde{\eta}_{\rm mo}(u_h) + \eta_{\rm cg}(u_h)$. We note that the constants $C^{\rm tr}$ and $C_{\T_h}$ are not known {\it a priori} and we use empirical estimates in the computation instead. 

%
As mentioned earlier, we do not assign the truncation residual to local contributions but compute it by \eqref{eqn:etatrc} directly and check its value during the adaptive process since it is much smaller than the total local estimator $\rho_{T}$ when $\Omega$ is sufficiently large (cf. \S~\ref{sec:local}). 

We give the main adaptive algorithm below. The algorithm essentially follows from \cite[Algorithm 3]{wang2018posteriori} but we adapt it to the current setting.

\begin{algorithm}[!ht]
\caption{A posteriori mesh refinement with control of the computational domain.}
\label{alg:size}
\begin{enumerate}
	\item[Step 0] \textit{Prescribe:} Set $\Omega$, $\T_h$, $N_{\max}$, $\rho_{\rm tol}$, $\tau_1$, $\tau_2$ and $R_{\max}$.
	
	\item[Step 1] \textit{Solve:} Solve the GRAC solution $u_{h}$ of \eqref{eq:min_ac} on the current mesh $\T_{h}$.  
	\item[Step 2] \textit{Estimate:} Carry out the stress tensor correction step in Algorithm \ref{alg:etc}, compute the error estimator $\rho_{T}$ by \eqref{eq:localestimator1} for each $T\in\T_h$, and the truncation error estimator $\eta_\trc$ by \eqref{eqn:etatrc}. Compute the degrees of freedom $N$ and $\rho = \sum_{T}\rho_T$. If $\eta_{\rm tr} > \tau_2 \rho$, enlarge the computational domain $\Omega$. Stop if $N>N_{\max}$ or $\rho < \rho_{\rm tol}$ or $R>R_{\max}$.
	
	\item[Step 3] \textit{Mark:} 
	\begin{enumerate}
	\item[Step 3.1]:  Choose a minimal subset $\M\subset \T_h$ such that
	\begin{displaymath}
		\sum_{T\in\M}\rho_{T}\geq\textrm{mean}(\rho).
	\end{displaymath}	 
	\item[Step 3.2]: We can find the interface elements which are within $k$ layers of atomistic distance, $\M^k_\i:=\{T\in\M\bigcap \T_h^\c: {\rm{dist}}(T, \Li)\leq k \}$. Choose $K\geq 1$, find the first $k\leq K$ such that 
	\begin{equation}
		\sum_{T\in \M^k_\i}\rho_{T}\geq \tau_1\sum_{T\in\M}\rho_{T},
		\label{eq:interface2}
	\end{equation}
	with tolerance $0<\tau_1<1$. If such a $k$ can be found, let $\M = \M\setminus \M^{k}_\i$. Then in step 3, expand the interface $\L^\i$ and the ``interface buffer" $\L^{\rm buf}$ outward by $k$ layers. 
	\end{enumerate}
	
	\item[Step 4] \textit{Refine:} Bisect all elements $T\in \M$. Go to Step 1.
\end{enumerate}
\end{algorithm}

%% file: numerics.tex
\section{Numerical Experiments}
\label{sec:numerics}

\newcommand{\fig}[1]{Figure \ref{#1}}
\newcommand{\tab}[1]{Table \ref{#1}}

In this section, we conduct the adaptive computations for the three types of defects, namely micro crack, anti-plane screw dislocation and multiple vacancies,  based on our adaptive algorithm proposed in Algorithm \ref{alg:size} with different {\it a posteriori} error estimators and mesh structures. The geometries of these defects are visualized in Figure \ref{fig:geom}.


We consider a two dimensional triangular lattice $\Lhom := \mathsf{A}\Z^2$ with 
\begin{eqnarray}
\mA=\Bigg(\begin{array}{cc}
    1 & 1/2 \\
    0 & \sqrt{3}/2
\end{array} \Bigg).
\end{eqnarray}
Note that upon rotating and possibly dilating, the projection of the BCC crystal is in fact a triangular lattice \cite{2013-defects} and hence it is of practical interest to be considered. 

We adopt the well-known EAM potential as the site energy potential $V_{\ell}$ through our numerical experiments, where
\begin{align}
  V_{\ell}(y) := & \sum_{\ell' \in \Nhd_{\ell}} \phi(|y(\ell)-y(\ell')|) + F\B(
  {\textstyle \sum_{\ell' \in \Nhd_{\ell}} \psi(|y(\ell)-y(\ell')|)} \B)\nonumber\\
    = &\sum_{\rho \in \Rg_{\ell}} \phi\b(|D_\rho y(\ell)|\b) + F\B(
  {\textstyle \sum_{\rho \in \Rg_{\ell}}} \psi\b( |D_\rho y(\ell)|\b) \B),  
    \label{eq:eam_potential}
\end{align}
for a pair potential $\phi$, an electron density function $\psi$ and
an embedding function $F$. In particular, we choose
\begin{displaymath}
\phi(r)=\exp(-2a(r-1))-2\exp(-a(r-1)),\quad \psi(r)=\exp(-br),
\end{displaymath}
\begin{displaymath}
\text{ and } \quad F(\tilde{\rho})=C\left[(\tilde{\rho}-\tilde{\rho_{0}})^{2}+
(\tilde{\rho}-\tilde{\rho_{0}})^{4}\right],
\end{displaymath}
with parameters $a=4, b=3, c=10$ and $\tilde{\rho_{0}}=6\exp(0.9b)$, which is exactly the same as the numerical experiments in \cite{COLZ2013} and \cite{wang2018posteriori}.


We test the performance of the {\it modified a posteriori} error estimators and the mesh structures with different parameters, i.e. the width of the buffer region and the width of the blending region, which we believe may influence the adaptive process. In particular, we carry out the adaptive computations for the {\it modified} elementwise local estimator $\rho_{T}(u_h;T)$ defined in \eqref{eq:localestimator1} first by using the {\it directly approximated} modeling residual $\tilde{\eta}_\mo(u_h; T)$ defined in \eqref{eq:direct_approx} with the widths of the buffer region being $R_{\rm buf}=1, 3$ and $10$. We then fix $R_{\rm buf}=3$ and apply the {\it blended approximated} modeling residual defined in \eqref{eq:approx} with the widths of the blending region $R_{\rm bld}=2$ and $6$. To better reveal the effectiveness and the efficiency of our error estimators and mesh structures, we compare the results of our adaptive computations with those using an {\it a priori} graded mesh \cite{2013-defects, PRE-ac.2dcorners}, the {\it original} residual based {\it a posteriori} error estimator given in \eqref{eq:localetam} and \eqref{eq:localetaC} and also the purely coarsening error estimator by totally neglecting the modeling residual.




In all numerical experiments, the radius of the computational domain $\Omega$ is set to be $300$ initially. The adaptive processes start with the initial configuration with $R_{\rm a}=6$, where for the case of multiple vacancies, it means that the initial radius of each disjoint atomistic region is 6. The adaptive parameters in Algorithm \ref{alg:size} are fixed to be $\tau_1=0.7$ and $\tau_2=1.0$. 


\subsection{Micro-crack}
\label{sec:numerics:m-crack}
The first defect we consider is the micro-crack, which is a prototypical example of point defects. We note that, technically speaking, this is not  a ``crack". It only serves as an example of a localized defect with an anisotropic shape. To generate this defect, we remove $k$ atoms from $\Lhom$,
\begin{align*}
\L_{k}^{\rm def}:=\{-(k/2)e_{1}, \ldots, (k/2-1)e_{1})\},     & \qquad{\rm if }\quad k \quad\textrm{is even},\\
\L_{k}^{\rm def}:=\{-(k-1)/2e_{1}, \ldots, (k-1)/2e_{1})\}, & \qquad{\rm if }\quad k \quad\textrm{is odd},
\end{align*}
and $\L = \Lhom\setminus \L_{k}^{\rm def}$. 
In our example, we set $k=11$. We apply an isotropic stretch $\mathrm{S}$ and a shear $\gamma_{II}$  by setting
\begin{displaymath}
{\sf B}=\left(
	\begin{array}{cc}
		1 & \gamma_{II} \\
		0            & 1+\mathrm{S}
	\end{array}	 \right)
	\cdot{\sf {F_{0}}},
\end{displaymath}
where ${\sf F_{0}} \propto \mathrm{I}$ is a macroscopic stretch or compression and $\mathrm{S}=\gamma_{II}=0.03$.

Figure \ref{fig:conv_mcrack_conv} shows the convergence of the true error $\|\nabla u_h - \nabla u\|_{L^2(\Omega)}$ with respect to the number of degrees of freedom $N$ and Figure \ref{fig:conv_mcrack_efffac} shows the efficiency factors (which is defined by the ratios of the error estimators and the true error) for various a posteriori error estimators and mesh structures. 

\begin{figure}[htb]
\centering
	\subfloat[Convergence]{
		\label{fig:conv_mcrack_conv}
		\includegraphics[height=5.5cm]{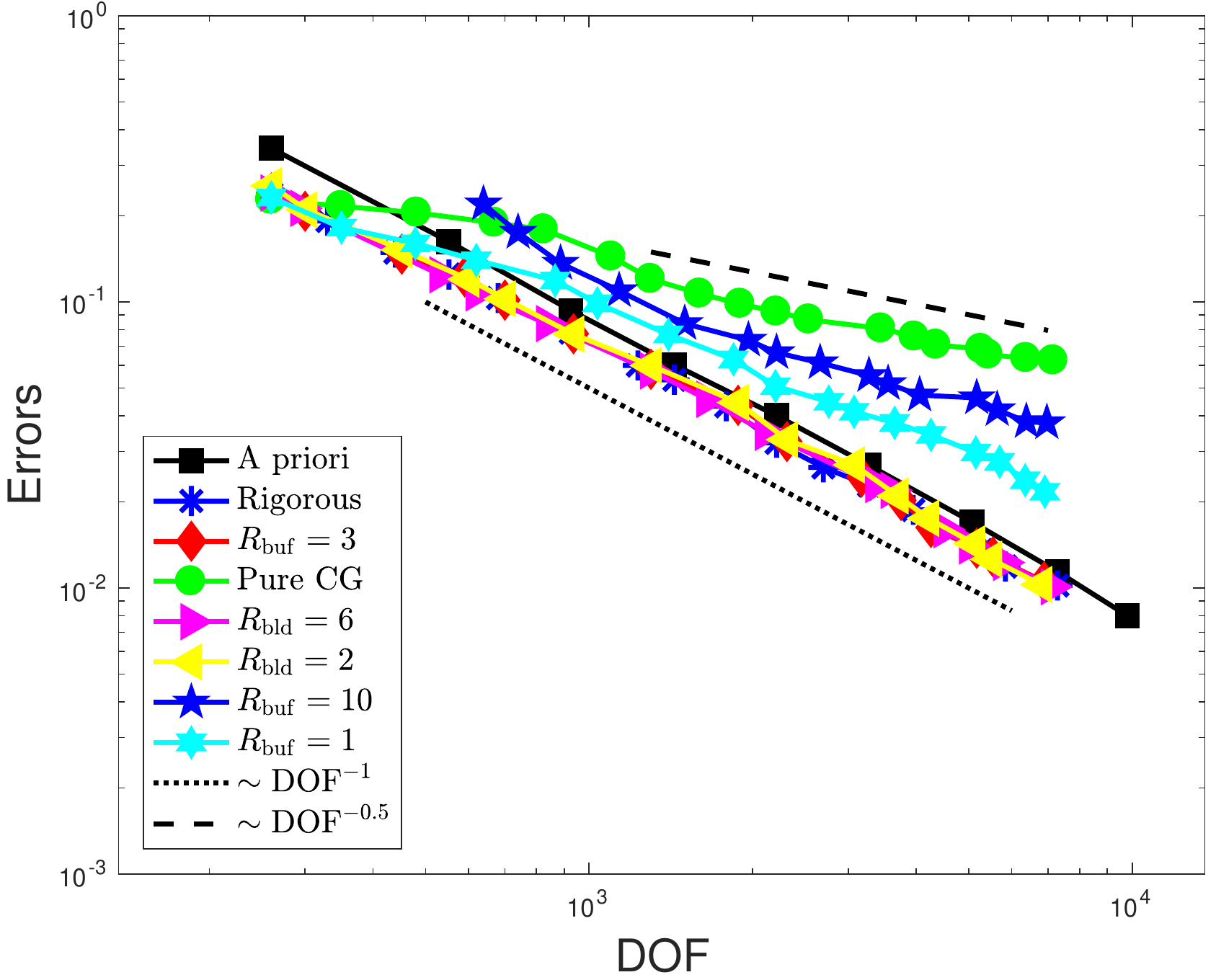}}
	\hskip0.2cm
	\subfloat[Efficiency factors]{
		\label{fig:conv_mcrack_efffac}
		\includegraphics[height=5.5cm]{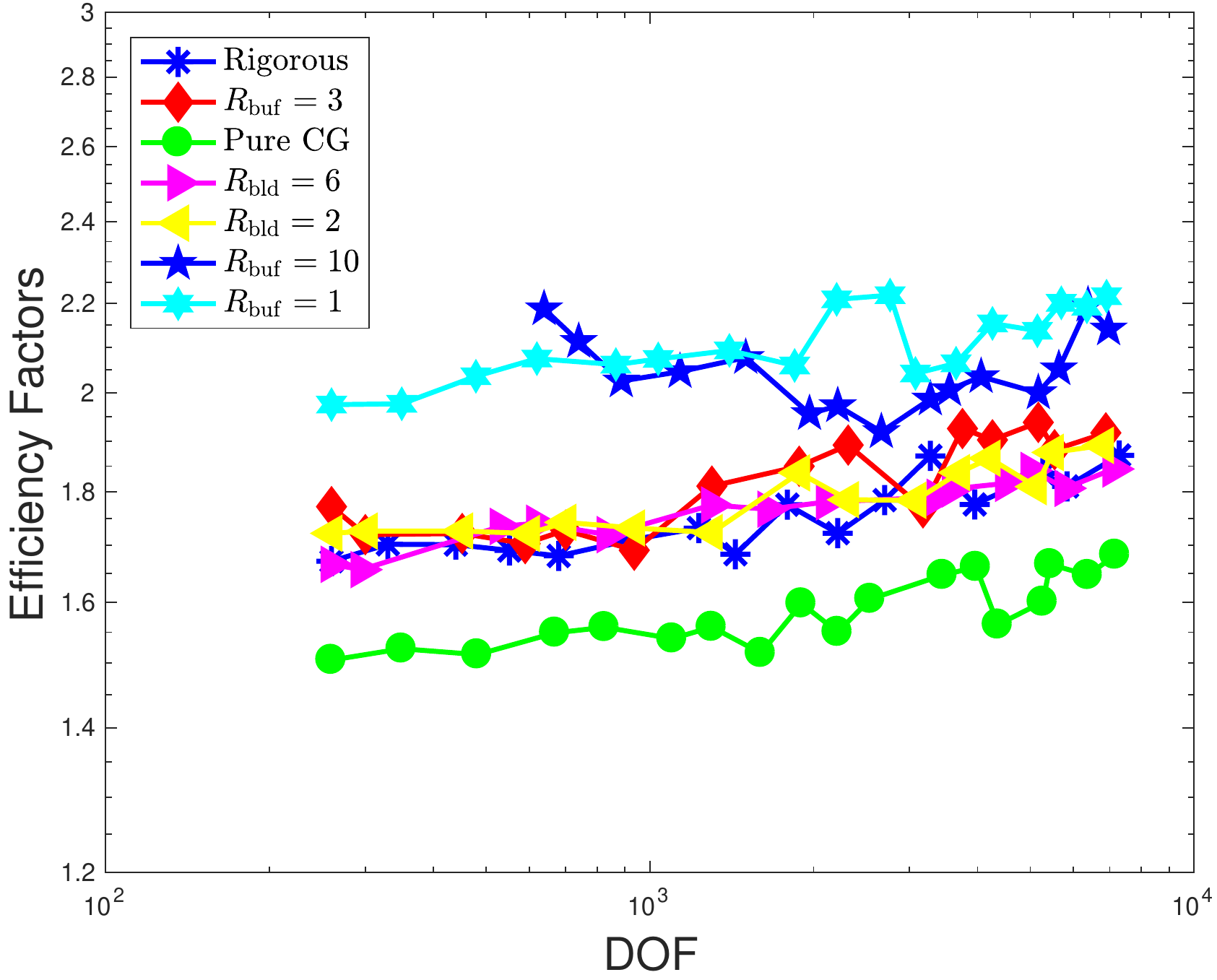}}
	\caption{The convergences of the error and the efficiency factors for different error estimators with respect to the number of degrees of freedom for the micro-crack.}
	\label{figs:conv_mcrack}
\end{figure} 

First of all, Figure \ref{fig:conv_mcrack_conv} clearly verifies the necessity of the prudence of treating the modeling residual estimator around the interface. In terms of the error versus the number of degrees of freedom, all adaptive computations achieve the same optimal convergence rate compared with the {\it a prior} graded mesh and the {\it original} residual based {\it a posteriori} error estimator and are barely distinguishable if we fix the width of the buffer region to be 3, i.e., $R_{\rm buf}=3$. Totally removing the modeling residual  from the {\it a posteriori} error estimator (represented by the green line with round markers) or crudely approximating it near the interface (represented by the light blue line with six-pointed star markers for $R_{\rm buf}=1$) apparently result in a lower convergence rate. However, we may not need to be too cautious either, since the influence of the modeling error decays rapidly \cite{dobson2009analysis, chen2012ghost, cui2013}. This is demonstrated by the dark blue line with star markers for $R_{\rm buf}=10$, where apparently
unnecessary degrees of freedom are placed around the interface that causes a even worse rate of convergence even compared with the crude approximation.
Figure \ref{fig:conv_mcrack_efffac} plots the efficiency factors for various error estimators and mesh structures. The factors are uniformly moderate for all the estimators with some oscillation which we believe has little effect for the adaptivity. We note that rigorous efficiency analysis, which proves that the residual based error estimator provides both upper and lower bounds for the true error up to some constants, requires a substantial amount of additional technicality and should be included in a separate work. We refer to \cite{wang2018analysis} for a theoretical consideration of the efficiency for the GRAC method in one dimension that may be the only existing literature in this field to the best knowledge of the authors. 


\begin{figure}[htp]
\begin{center}
	\subfloat[CPU times]{\label{fig:time_mcrack}\includegraphics[height=5.5cm]{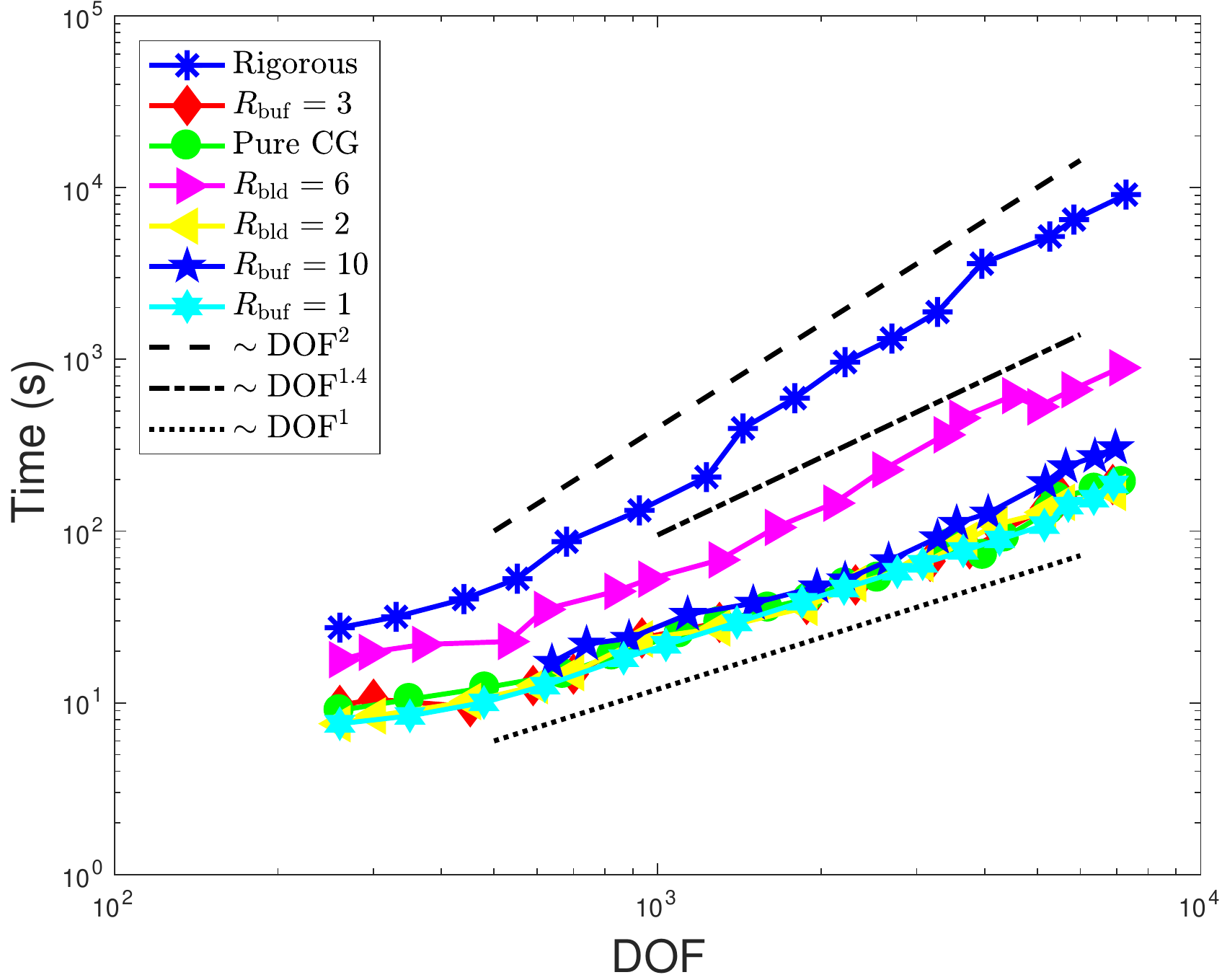}}
	\hskip0.2cm
	\subfloat[Time constitution]{\label{fig:time_mcrack_sp}\includegraphics[height=5.5cm]{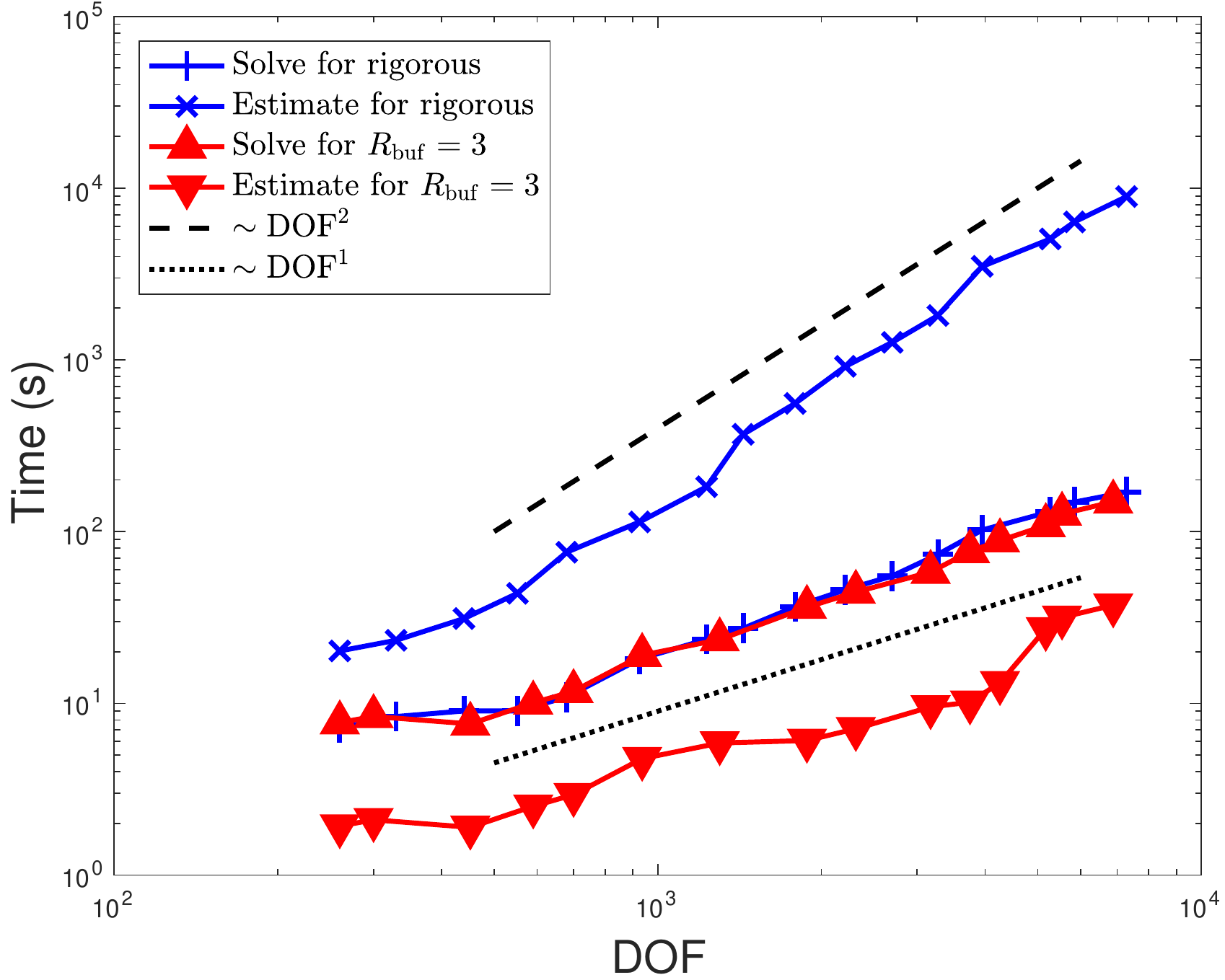}}
	\caption{The CPU times for each step in the adaptive computations for the micro-crack. }
	\label{fig:time_mcrack_all}
\end{center}
\end{figure}


The most interesting observation that illustrates the significance of the current work can be concluded in Figure \ref{fig:time_mcrack_all}.  Figure \ref{fig:time_mcrack} plots the total CPU times versus the numbers of degrees of freedom $N$ for various error estimators and mesh structures while Figure \ref{fig:time_mcrack_sp} plots those for computing the error estimators and solving the a/c coupling problems separately. We see clearly in Figure \ref{fig:time_mcrack} that the CPU time of adaptively solving the a/c problem using the {\it original a posteriori} error estimator scales as $O(N^2)$ whereas those using all other estimators scales as $O(N)$, except for the case of $R_{\bld} = 6$ (a slightly large blending region). This coincides with our estimate of the computational cost in \S~\ref{sec:local} and evidently shows that the {\it original a posteriori} error estimator (represented by the blue line with asterisk markers) may not be employed for the purpose of practical adaptive computations. It is also interesting to see that although we observe in Figure \ref{fig:conv_mcrack_conv} an optimal convergence rate of the error using the blended approximated estimator, Figure \ref{fig:time_mcrack} shows that such construction may gild the lily, even when the blending region is just slightly large (represented by the purple line with triangle markers for $R_{\rm bld}=6$). It may be explained by the fact that all the computations for the modeling residual, including computing the atomistic stress tensor and finding the geometric relation between  $\Ta$ and $\T_h$, are still carried out in the blending region which increase the computational cost rapidly with the growing width of the blending region. The dark blue line with star markers, which represents the mesh structure having a large buffer region ($R_{\rm buf}=10$), is noticed to be slightly higher than other ``$O(N)$" curves. This shows that computing the modeling residual is still more expensive than the coarsening residual for the reason that the modeling residual is ultimately nonlocal. This indicates that we may avoid computing the atomistic stress tensor as much as possible if the interface effect is properly captured for the purpose of efficiency. 

A more straightforward comparison of the {\it original} error estimator and the {\it modified} one is given in Figure \ref{fig:time_mcrack_sp}. It shows that the additional work purely comes from the computing the {\it original} estimator but the cost of computing the {\it modified} estimator is somehow marginal compared with that of solving the a/c problems themselves.

\subsection{Anti-plane Screw Dislocation}
\label{sec:numerics:screw dislocation}

The second defect for which we conduct the adaptive computation using Algorithm \ref{alg:size} is the anti-plane screw dislocation. Following \cite{2013-defects}, we restrict ourselves to an anti-plane shear motion which is illustrated in Figure \ref{fig:geom_screw}. In this case, the Burgers vector is set to be $b=(0,0,1)^T$ and the center of the dislocation core is chosen to be $\hat{x}:=\frac{1}{2}(1,1,\sqrt{3})^T$, and we assume that no additional shear deformation is applied. Thus the unknown for this dislocation model is the displacement in the $e_3$-direction. As we mentioned at the beginning of \S~\ref{sec:formulation:atm}, we choose $u_0 \in \Us$ given by the solution of a linear elasticity model as a far-field crystalline environment or a {\it predictor} in the language of numerical computation and we need to compute $u$ which can be considered as a {\it corrector}. For the sake of completeness, we review the detailed derivation of $u_0$ in \ref{sec:appendix}.

Similar results are observed for the anti-plane screw dislocation as those in the case of micro-crack with half of the convergence rate as predicted in \cite{2013-defects}. It is interesting to see that the adaptive computation using the pure coarsening residual (represented by the green line with round markers) again results in half of the optimal convergence rate (this time being DOF to the power of $0.25$ rather than to the power of $0.5$). Whether it is a coincidence needs a further investigation.

%
%
%

\begin{figure}[htb]
\centering
	\subfloat[Convergence]{
		\label{fig:conv_screw_conv}
		\includegraphics[height=6cm]{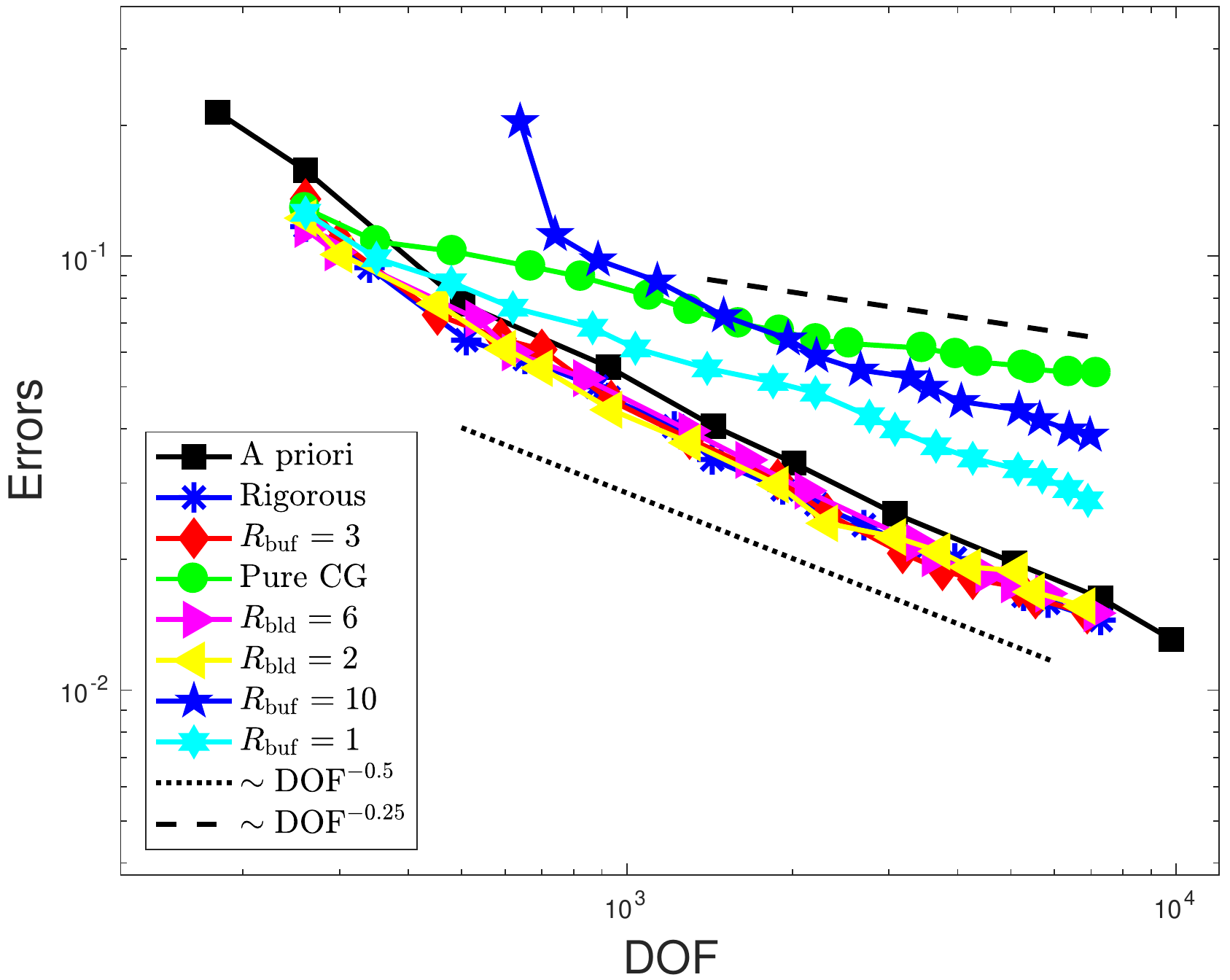}}
	\hskip0.2cm
	\subfloat[Efficiency factors]{
		\label{fig:conv_screw_efffac}
		\includegraphics[height=6cm]{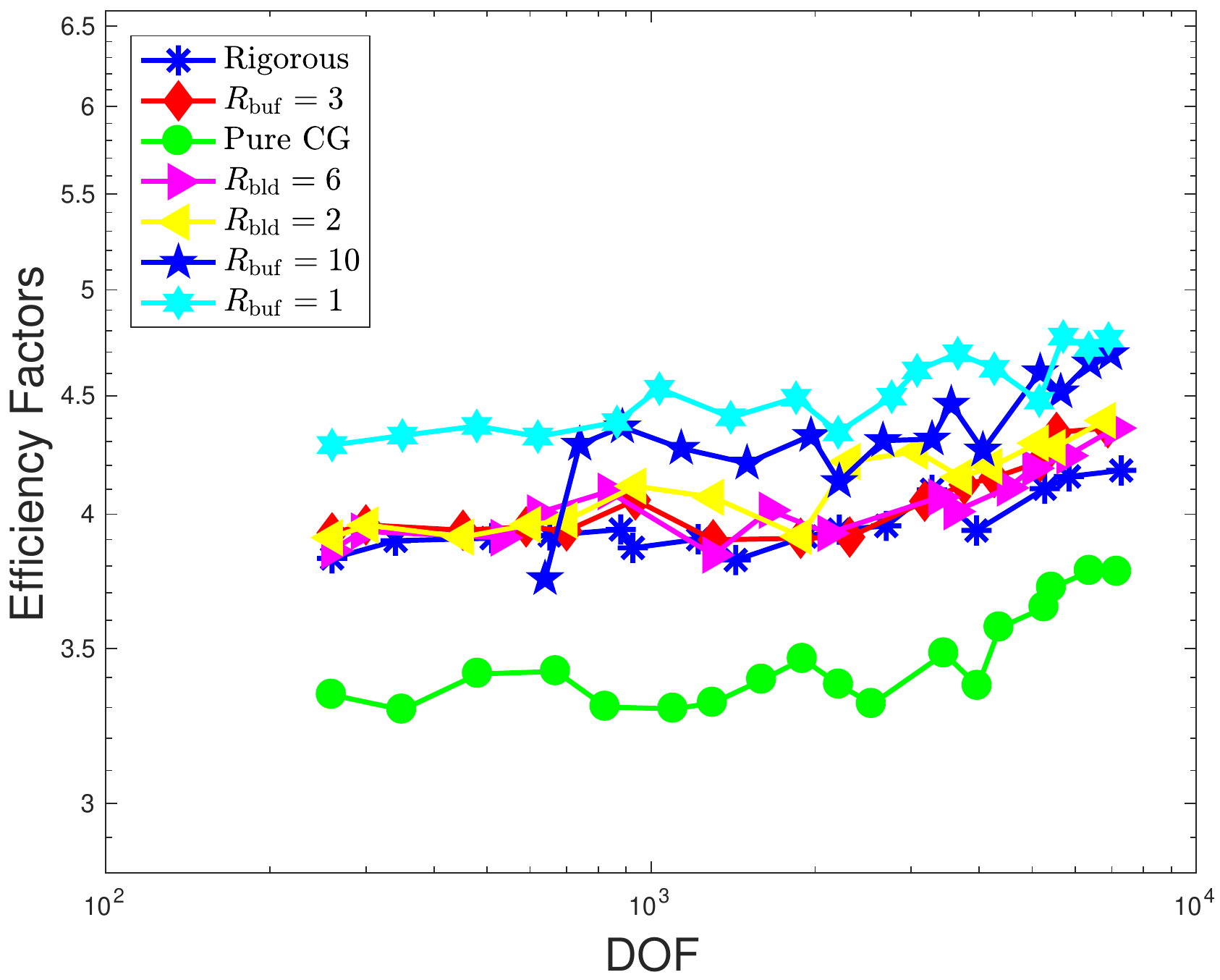}}
	\caption{The convergences of the error and the efficiency factors for different error estimators with respect to the number of degrees of freedom for the anti-plane screw dislocation.}
	\label{figs:conv_screw}
\end{figure} 

\begin{figure}[htp]
\begin{center}
	\subfloat[CPU times]{\label{fig:time_screw}\includegraphics[height=5.5cm]{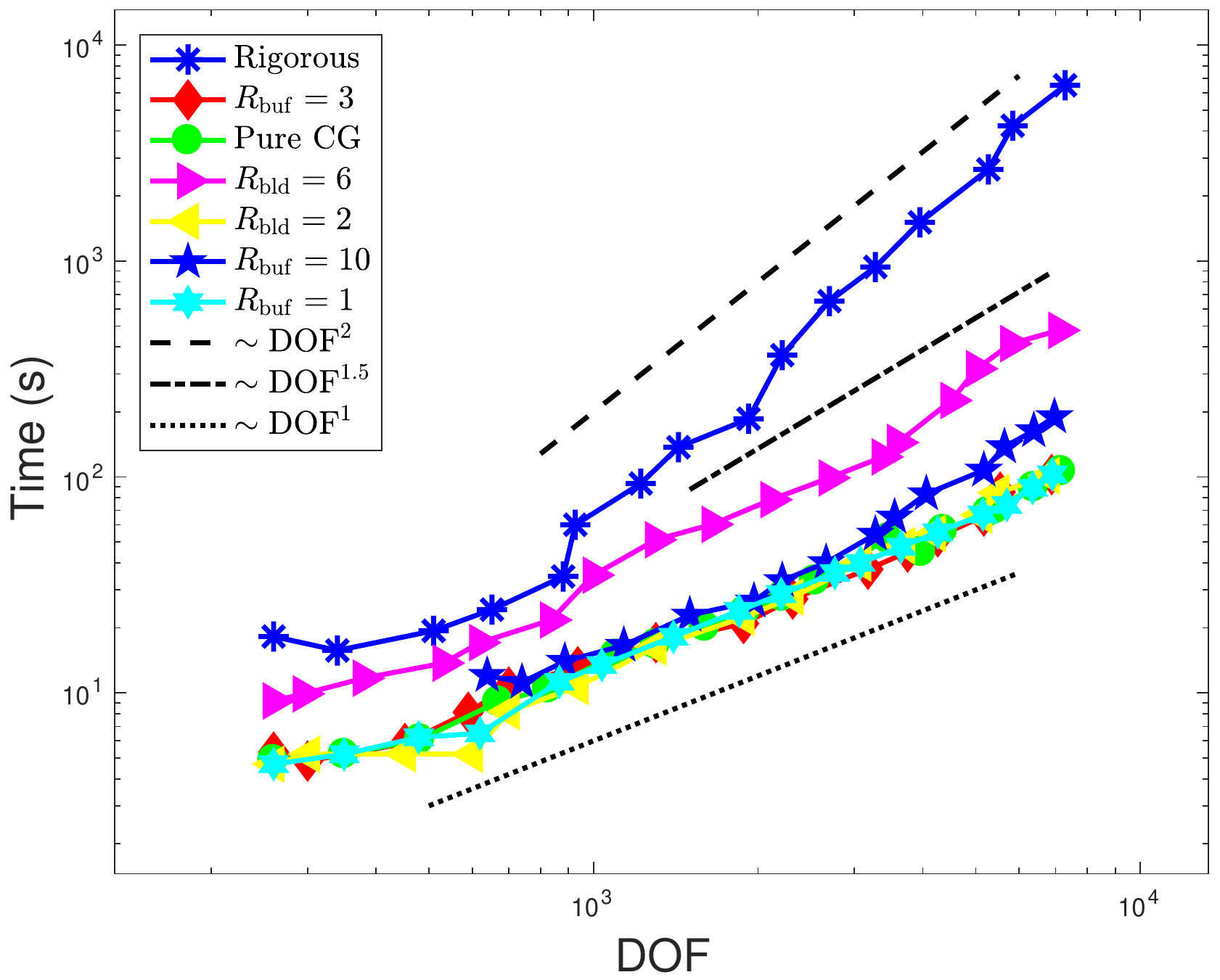}}
	\hskip0.3cm
	\subfloat[Time constitution]{\label{fig:time_screw_sp}\includegraphics[height=5.5cm]{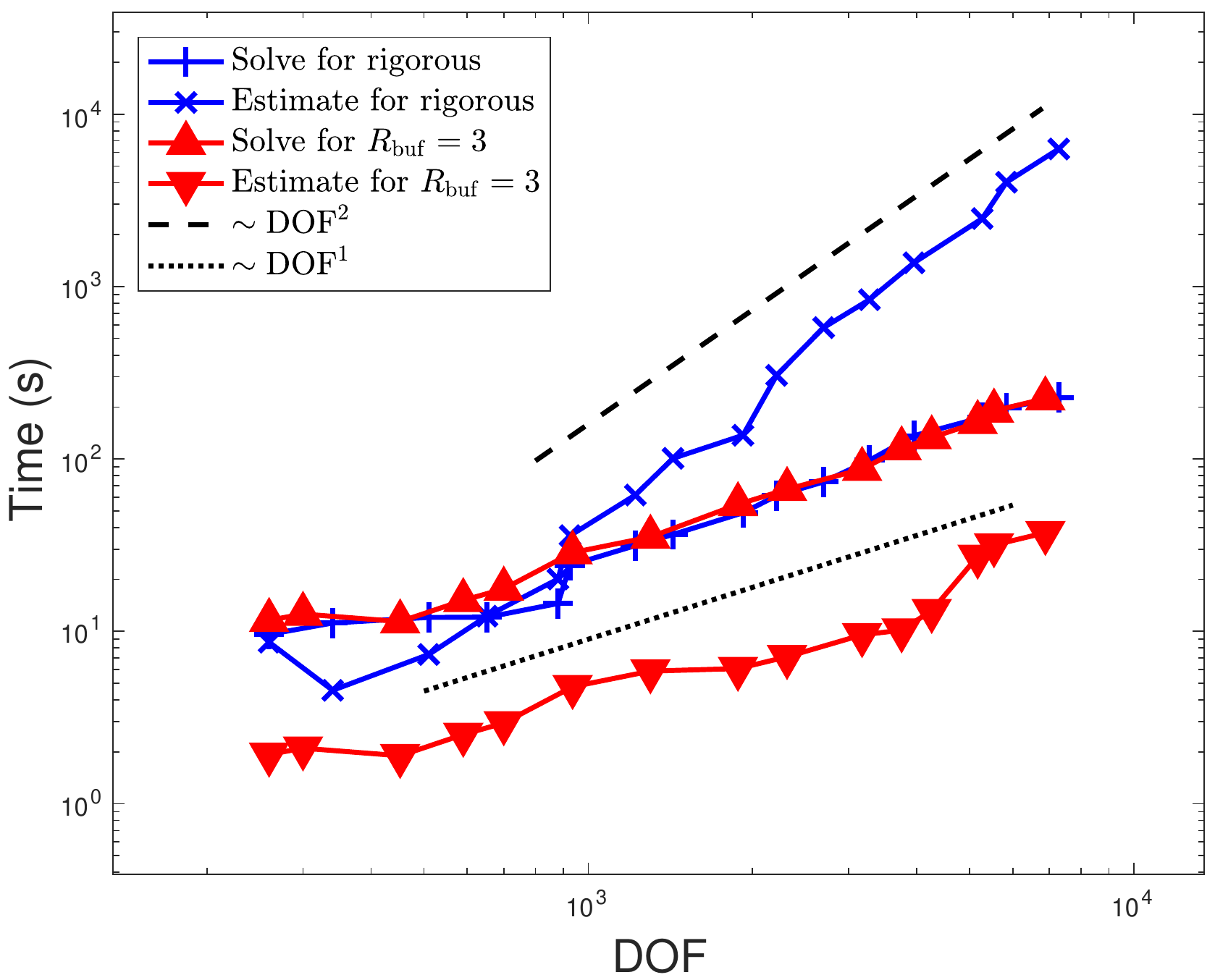}}
	\caption{The CPU times for each step in the adaptive computations for the anti-plane screw dislocation. }
	\label{fig:time_screw_all}
\end{center}
\end{figure}

\subsection{Multiple Vacancies }
\label{sec:numerics:multiple vacancies}
The last type of defect we consider is the case of multiple vacancies. We select three vacancy sites which are separated away from each other by a distance of $40$ (see Figure \ref{fig:geom_multivac} for an illustration). We then apply the same isotropic stretch $\mathrm{S}$ and shear $\gamma_{II}$ to this defected system as that in the micro-crack. This is a very simple model for the interactions of defects and we choose this case to see if the adaptive a/c coupling method can be potentially applied to study such problems. 

We show in Figure \ref{figs:evolution} the evolution of the atomistic and continuum partitioning in the adaptive process. Unlike the micro-crack and the anti-plane screw dislocation where only one atomistic region exists throughout, we start with three disjoint atomistic regions initially which contain the three separated vacancies. The atomistic regions then grow and get very close as the adaptivity goes and finally merge into one ``big" atomistic region. We see in the fourth subplot Figure \ref{figs:evolution} that the sizes of the mesh between the atomistic regions also shrink so small that we should adaptively shift the continuum model to the atomistic model. 

%
%

\begin{figure}[htb]
\begin{center}
	\includegraphics[scale=0.6]{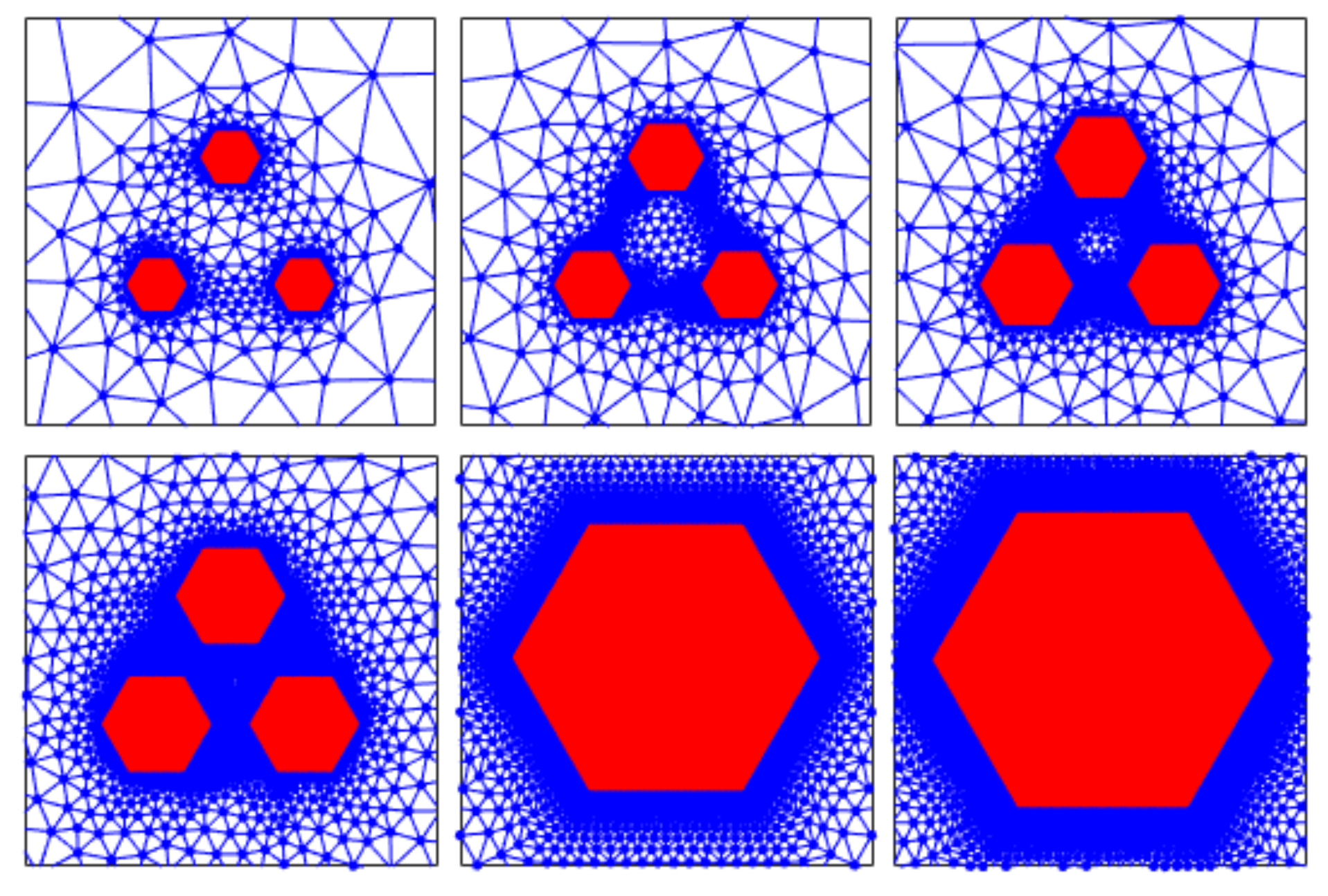}
	\caption{Illustration of the evolution of atomistic and continuum partition in the adaptive process for the multiple vacancies}
	\label{figs:evolution}
\end{center}
\end{figure}

We observe in Figure \ref{fig:conv_multivac_conv} a sudden drop of the error in each line which corresponds to the merge of the disjoint atomistic regions though the lines all possess their own ``correct" convergence rate. We put the term ``correct" in a quotation mark because there is actually no predicted convergence rate for this case before the merge of the atomistic regions. This is also the reason why there is no line representing the graded mesh, whose construction is often based on such prediction, is plotted in Figure \ref{fig:conv_multivac_conv} as opposed in Figure \ref{fig:conv_mcrack_conv} and Figure \ref{fig:conv_screw_conv}. However, this may be a more common situation we need to face when dealing with realistic material systems. It is interesting to discuss the issue of how we might adjust the adaptive strategy so that the time spot of the merge of the atomistic is (quasi-)optimal though it is beyond the scope of the current work and may be included in a future numerical study of the adaptive a/c method.

\begin{figure}[htb]
\centering
	\subfloat[Convergence]{
		\label{fig:conv_multivac_conv}
		\includegraphics[height=5.5cm]{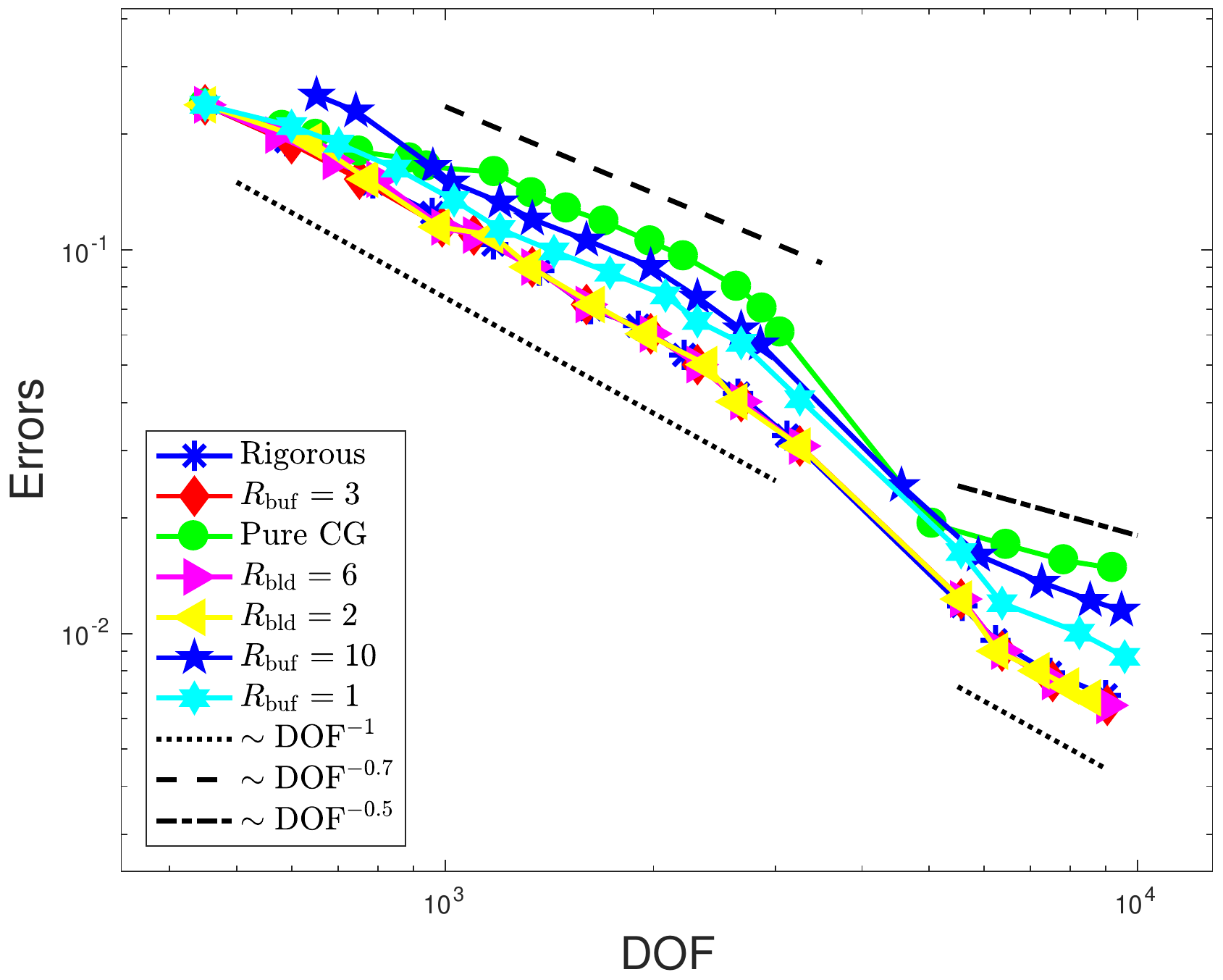}}
	\hskip0.2cm
	\subfloat[Efficiency factors]{
		\label{fig:conv_multivac_efffac}
		\includegraphics[height=5.5cm]{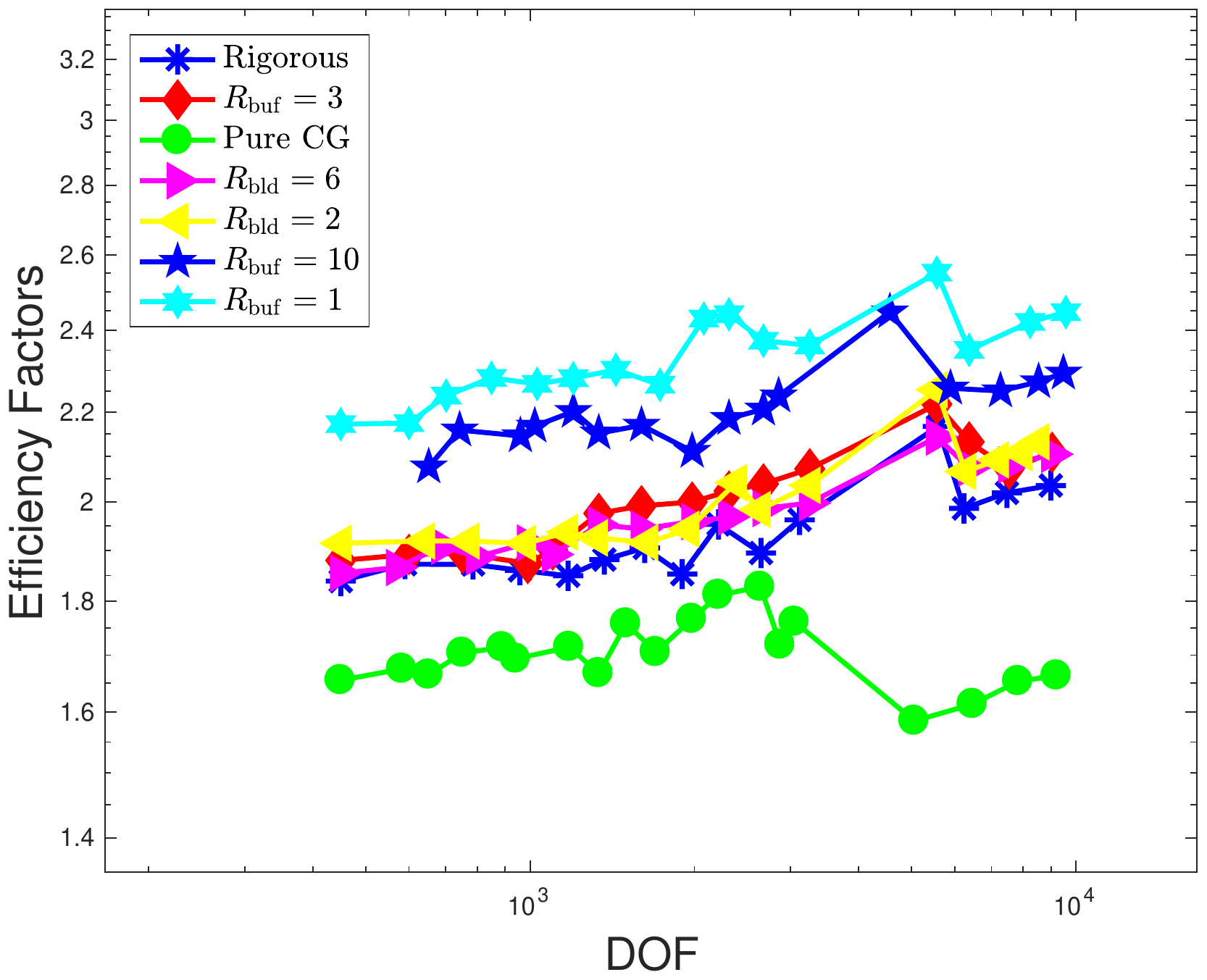}}
	\caption{The convergences of the error and the efficiency factors for different error estimators with respect to the number of degrees of freedom for the multiple vacancies.}
	\label{figs:conv_multivac}
\end{figure}

\begin{figure}[htp]
\begin{center}
	\subfloat[CPU times]{\label{fig:time_multivac}\includegraphics[height=5.5cm]{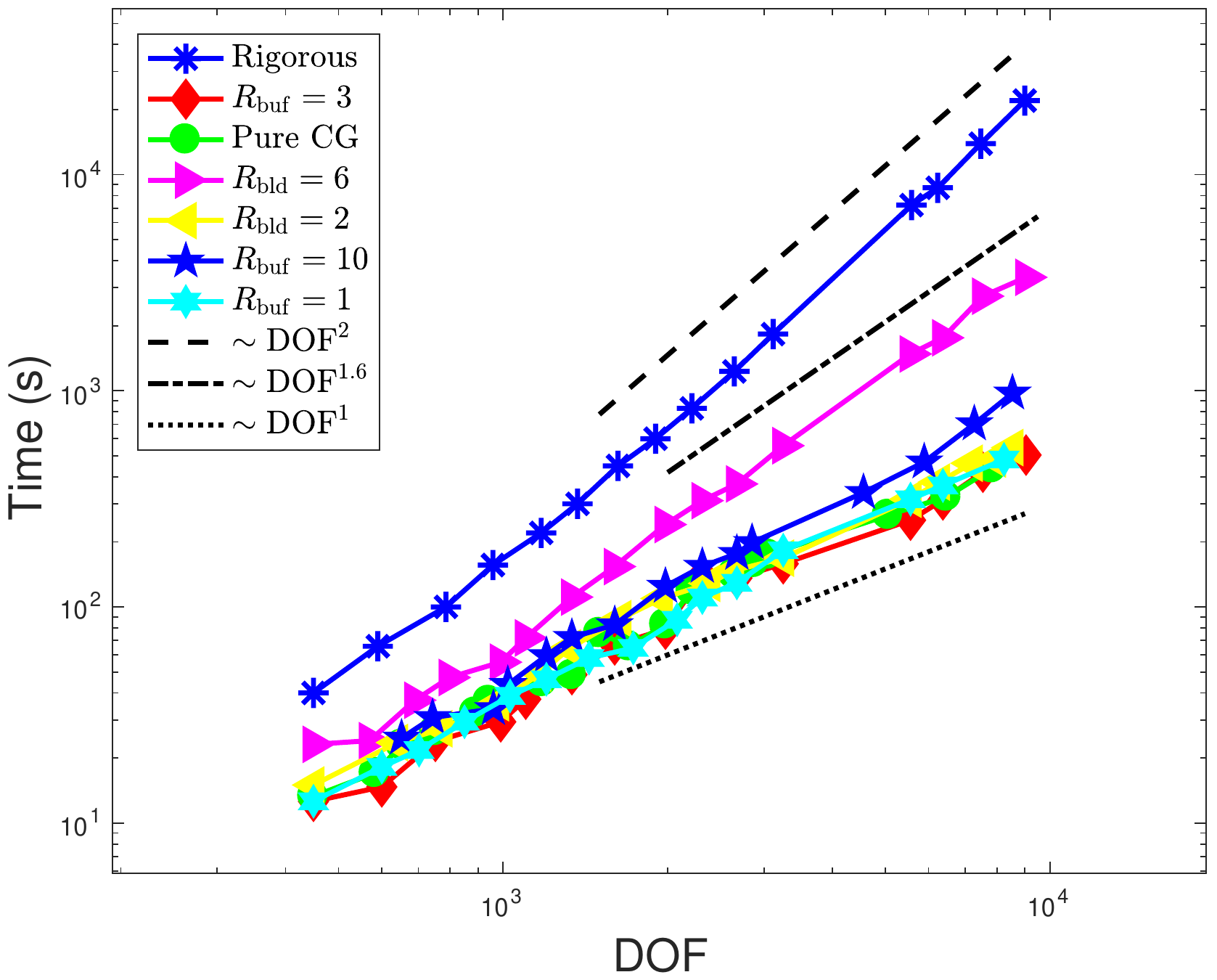}}
	\hskip0.3cm
	\subfloat[Time constitution]{\label{fig:time_multivac_sp}\includegraphics[height=5.5cm]{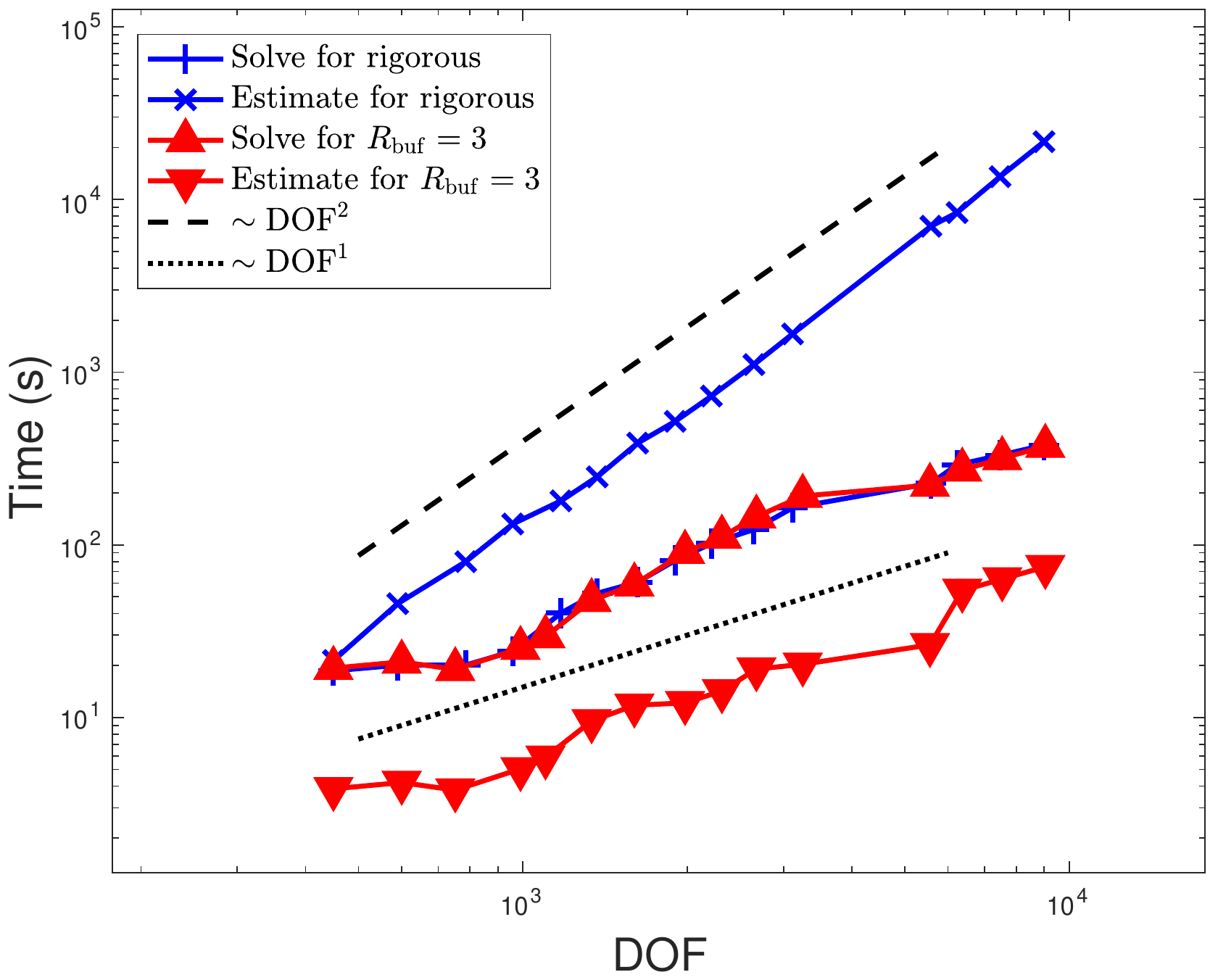}}
	\caption{The CPU times for each step in the adaptive computations for the multiple vacancies. }
	\label{fig:time_multivac_all}
\end{center}
\end{figure}

The final figure in this paper once again shows the advantage of using the properly {\it modified a posteriori} error estimator. We observe a slightly worse rate for the blended estimator (the purple line with triangle markers with the rate of $1.6$). This might be caused by the complex geometry between the atomistic regions, which confirms our belief that any unnecessary search for the geometric relationship between $\Ta$ and $\T_h$ should be avoided.

%% file: conclusion.tex
\section{Conclusion and Outlook}
\label{sec:conclusion}


We develop the efficient adaptive algorithm for the geometric reconstruction based consistent atomistic/continuum (GRAC) coupling scheme based on a {\it modified} {\it residual based} {\it a posteriori} error estimator for various types of crystalline defects. The key for achieving such efficiency but maintaining comparable accuracy is the strategy of constructing an a/c coupling mesh with a small ``interface buffer" region around the coupling interface in which the elementwise modeling residual is computed exactly but outside which the modeling residual is replaced by a theoretically and numerically justified approximation using its coarsening counterpart. Although we believe such strategy is generally applicable for a/c coupling schemes and more complex crystalline defects, the current research still raises a few open problems which deserve further mathematical analysis and algorithmic developments.

%

The first one is the extension to general short range interactions \cite{liao2018posteriori}.  To generalize our strategy to that setting, a detailed study of the width of ``interface buffer" region with respect to the interaction range is needed. 


The second one is the problem of adaptive error control for energy blended a/c coupling schemes, such as energy-based blended quasi-continuum (BQCE) method and blended ghost force correction (BGFC) method. The {\it a priori} analysis in \cite{LiOrShVK:2014, OrZh:2016, MiLu:2011} provide an analytical framework and the stress tensor based formulation plays a key role in the analysis. By inheriting such framework and using the {\it modified} error estimator proposed in this paper, the adaptivity for those coupling schemes should be somehow straightforward for the case of fixed width of the blending region. However, a considerable amount of effort may be required to problem of the adaptive choice of the width of the blending region.


The third one is the extension to three dimensional problems. We note that constructing a consistent a/c coupling scheme with a sharp interface in three dimensions still remains an open problem. Thus the development of adaptive energy blended methods mentioned in the previous paragraph may be the appropriate direction.


The last but may be the most important problem to consider is the adaptive computation for complex crystalline defects. The simulation of the interactions of defects, which is briefly touched by the third numerical experiment, and the dislocation nucleation have already attracted considerable attention. It is very difficult (if not impossible) to give rigorous {\it a priori} analysis for such problems, including the proper boundary conditions and complicated geometry at the interface. However, it should be the place where the adaptive a/c coupling methods reveal its most advantage and power and we believe that the idea of the current research makes it possible to develop efficient and robust adaptive algorithms for those problems with practical interests.


%% file: appendix.tex

\section{Far-field boundary predictor for anti-plane screw dislocations}
\label{sec:appendix}

\newcommand{\ulin}{u^{\rm lin}}
\newcommand{\burg}{{\sf b}}

We model dislocations by following the setting used in \cite{2013-defects}. We consider a model for straight dislocations obtained by projecting a 3D crystal into 2D. Let $B \in \R^{3\times 3}$ be a nonsingular matrix. Given a Bravais lattice $B\Z^3$ with dislocation direction parallel to $e_3$ and Burgers vector $\burg=(\burg_1,0,\burg_3)$, we consider
displacements $W: B\Z^3 \rightarrow \R^3$ that are periodic in the direction of the dislocation direction of $e_3$. Thus, we choose a projected reference lattice $\L := A\Z^2 := \{(\ell_1, \ell_2) ~|~ \ell=(\ell_1, \ell_2, \ell_3) \in B\Z^3\}$. We also introduce the projection operator 
\begin{equation}\label{eq:P}
    P(\ell_1, \ell_2) = (\ell_1, \ell_2, \ell_3) \quad \text{for}~\ell \in B\Z^3.
\end{equation}
It can be readily checked that this projection is again a Bravais lattice. For anti-plane screw dislocation, $\L$ is obtained as projection of a 3D Bravais lattice along the screw dislocation direction (and the direction of slip) and we restrict the displacements of the form $u = (0, 0, u_3)$.

We follow the constructions in \cite{2013-defects, 2017-bcscrew} for modeling dislocations and prescribe $u_0$ as follows. Let $\L\subset\R^2$, $\hat{x}\in\R^2$ be the position of the dislocation core and $\Upsilon := \{x \in \R^2~|~x_2=\hat{x}_2,~x_1\geq\hat{x}_1\}$ be the ``branch cut'', with $\hat{x}$ chosen such that $\Upsilon\cap\Lambda=\emptyset$.

We define the far-field predictor $u_0$ by solving the continuum linear elasticity (CLE)
\begin{eqnarray}\label{CLE}
\nonumber
\mathbb{C}^{j\beta}_{i\alpha}\frac{\partial^2 u^{\rm lin}_i}{\partial x_{\alpha}\partial x_{\beta}} &=& 0 \qquad \text{in} ~~ \R^2\setminus \Upsilon,
\\
u^{\rm lin}(x+) - u^{\rm lin}(x-) &=& -\burg \qquad \text{for} ~~  x\in \Upsilon \setminus \{\hat{x}\},
\\
\nonumber
\nabla_{e_2}u^{\rm lin}(x+) - \nabla_{e_2}u^{\rm lin}(x-) &=& 0 \qquad \text{for} ~~  x\in \Upsilon \setminus \{\hat{x}\},
\end{eqnarray}
where the forth-order tensor $\mathbb{C}$ is the linearised Cauchy-Born tensor (derived from the potential $V$, see \cite[\S~7]{2013-defects} for more detail).

We mention that for the anti-plane screw dislocation, under the proper assumptions on the interaction range $\Rg$ and the potential $V$, the first equation in \eqref{CLE} simply becomes to $\Delta u^{\rm lin} = 0$ \cite{2017-bcscrew}. The system \eqref{CLE} then has the well-known solution 
\begin{align}\label{predictor-u_0-dislocation}
u_0(x) := u^{\rm lin}(x) = \frac{\burg}{2\pi}\arg(x-\hat{x}),
\end{align}
where we identify $\R^2 \cong \C$ and use $\Upsilon-\hat{x}$ as the branch cut for arg.

Note that for the purpose of analysis, we have $\nabla u_0 \in C^{\infty}(\R^2\setminus\{0\})$ and $|\nabla^j u_0| \leq C|x|^{-j}$
for all $j \geq 0$ and $x \neq 0$.